\RequirePackage{fix-cm}
\documentclass[smallextended]{svjour3}       % onecolumn (second format)
\smartqed  % flush right qed marks, e.g. at end of proof
\usepackage{amsfonts,comment}
\usepackage{amssymb,amsmath,url,bm}
\usepackage{mathrsfs}
\usepackage{booktabs}
\usepackage{hyperref}
\usepackage{graphicx,color,float,epstopdf}
\usepackage{mathrsfs}
\usepackage{pst-blur,pstricks-add}
\usepackage{cases}
\usepackage{algorithm}
\usepackage{algorithmicx}
\usepackage{algpseudocode}
\usepackage{graphics,booktabs,epsfig,subfigure}
\usepackage[numbers,sort&compress]{natbib}
\usepackage{multirow}
\usepackage{geometry}
\DeclareMathOperator{\prox}{Prox}

\newtheorem{ass}[theorem]{Assumption}
\newcommand{\h}[1]{\mathbf{#1}}

\newcommand{\nn}{\nonumber}

\newcommand{\R}{\mathbb{R}}

\DeclareMathOperator*{\dom}{dom}
\DeclareMathOperator*{\inte}{int}
\begin{document}

\title{A Single-loop Proximal Subgradient Algorithm  for A Class  Structured  Fractional Programs\thanks{This work was supported by the Ministry of Science and Technology of China (No.2021YFA1003600), and the NSFC grants 12131004,
12471289, and the Jiangsu University QingLan Project.
}}


\author{ {\bf \normalsize Deren Han}\thanks{School of Mathematical Sciences, Beihang University, Beijing 100191, P.R. China. Email: \texttt{\color{blue}handr@buaa.edu.cn}} \and
	{\bf \normalsize Min Tao}\thanks{School of Mathematics, National Key Laboratory for Novel Software Technology, Nanjing University, Nanjing 210093, P.R. China. Email: \texttt{\color{blue}taom@nju.edu.cn}} \and
	{\bf \normalsize Zihao Xia}\thanks{School of Mathematics, Nanjing University, Nanjing 210093, P.R. China. Email: \texttt{\color{blue}zihaoxia@smail.nju.edu.cn}}
}


%\institute{}

\date{Received: date / Accepted: date}

\maketitle

\begin{abstract}

In this paper, we investigate a class of nonconvex and nonsmooth fractional programming problems, where the numerator composed of two parts: a convex, nonsmooth function and a differentiable, nonconvex function, and the denominator consists of a convex, nonsmooth function composed of a linear operator. These structured fractional programming problems have broad applications, including CT reconstruction, sparse signal recovery, the single-period optimal portfolio selection problem and standard Sharpe ratio minimization problem.
We develop a single-loop proximal subgradient algorithm that alleviates computational complexity by decoupling the evaluation of the linear operator from the nonsmooth component. We prove the global convergence of the proposed single-loop algorithm to an exact lifted stationary point under the Kurdyka-\L ojasiewicz assumption. Additionally, we present a practical variant incorporating a nonmonotone line search to improve computational efficiency. Finally, through extensive numerical simulations, we showcase the superiority of the proposed approach over the existing state-of-the-art methods for three applications: $L_{1}/S_{\kappa}$ sparse signal recovery, limited-angle CT reconstruction, and optimal portfolio selection.

\noindent\textbf{keywords:} fractional programming,  single-loop,
	convergence analysis, decoupling

 \end{abstract}

\section{Introduction}
In this paper, we consider the following class of nonsmooth and  nonconvex fractional programs:
\begin{eqnarray}\label{PForm} \min_{\h x\in{\cal S}} F({\h x}):=\frac{g({\h x})+h({\h x})}{f({K}{\h x})}, \end{eqnarray}
where $g:{\mathbb R}^m\rightarrow{\overline{\mathbb R}}$,  $f:{\mathbb R}^p\rightarrow{\overline{\mathbb R}}$ are proper, convex, lower semicontinuous functions, $h:{\mathbb R}^n\rightarrow{{\mathbb R}}$ is a (possibly nonconvex) differentiable function over an open set containing ${\cal S}$  and its gradient is Lipschitz continuous with a  Lipschitz constant $L_{\nabla h}$. Additionally, ${K}: {\mathbb R}^{n}\rightarrow {\mathbb R}^p$ is linear operator. Furthermore, to ensure that (\ref{PForm}) is well-defined, we also require $g({\h x})+h({\h x})\geqslant0$ and $f({K}{\h x}) > 0$ for all $\h x\in\mathcal{S}$.

Model (\ref{PForm})  captures a range of optimization problems across diverse areas, such as the limited-angle CT reconstruction problem \cite{WTNL}, the first phase of single-period optimal portfolio selection problem involving a risk-free asset \cite{pang1980parametric}, the minimization of the Sharpe ratio \cite{CHZ11}, and the scale-invariant sparse signal reconstruction problem \cite{Tao22, ZengYuPong20,LSZ}. Here we provide several specific examples to illustrate the inherent versatility of the model (\ref{PForm}).
\begin{itemize}
\item[(a)]  {\bf $L_{1}/S_{\kappa}$ sparse signal recovery.}
Considering the following sparse signal recovery with the sparsity-promoting regularizer $L_{1}/S_{\kappa}$ \cite{LSZ}:
\begin{equation}
	\min_{\h x\in\mathcal{S}}\ \frac{\lambda||\h x||_{1} + \frac{1}{2}||A\h x - \h b||_{2}^{2}}{||\h x||_{(\kappa)}} \label{L1oLK}
\end{equation}
where ${\h x}\in\mathbb R^n$ and for a fix $\kappa\in[n]$, $||\h x||_{(\kappa)}$ represents the sum of the absolute values of the largest $\kappa$ elements. The matrix $A$ is usually generated by Gaussian distribution or oversampled discrete cosine transform, and $\h b$ is the observation vector and
$\mathcal{S} = [{\h c},{\h d}]^{n}$.
By setting $h(\h x) = \frac{1}{2}||A\h x - \h b||_{2}^{2}$, $g(\h x) = \lambda||\h x||_{1}$, $f(\h x) = ||\h x||_{(\kappa)}$, $K =  I$, (\ref{L1oLK}) falls within the framework described in (\ref{PForm}).

\item[(b)] {\bf Limited-angle CT reconstruction}. Consider the following CT reconstruction problem model \cite{boct2023full}:
\begin{equation}
	\min_{{\h x}\in{\cal S}}\ \frac{\lambda||\nabla\h x||_{1} + \frac{1}{2}||A\h x - \h b||_{2}^{2}}{||\nabla\h x||_{2}} \label{CT}
\end{equation}
%In \ref{CT} $\h x$ denotes a 2D image defined on $m\times n$ Cartesian grid, and we adopt a linear index for it, i.e., $\h x_{ij}$ is the ((i-1)*m + j)'th component of $\h x$. The operator A is generated as the discrete Radon transform with the same resolution as $\h x$. $\nabla$ denotes the discrete gradient operator:
where $\nabla: \mathbb{R}^{n\times n}\rightarrow \mathbb{R}^{n\times n} \times \mathbb{R}^{n\times n}$ is the discrete gradient operator and the linear operator $A $ is the discrete Radon transform. Regarding the box constraint  ${\cal S} := [{\h c}, {\h d}] \subseteq \mathbb{R}^{n \times n}$, representing the box constraint on the reconstructed image \cite{WTNL}, we assume that ${\cal S} \cap {\text{span}}({\bf E}) =\emptyset$, where ${\bf E}$ is the matrix with all entries equal to one.
Problem (\ref{CT}) can be considered as a special case of the model (\ref{PForm}) by setting $h(\h x) = \frac{1}{2}||A\h x - \h b||_{2}^{2},\ g(\h x) = \lambda||\nabla\h x||_{1}, \ f(\h x) = ||\h x||_{2},\ K = \nabla$.

%(2) Optimal Portfolio Selection Problem. The single-period optimal portfolio selection problem is the following: An investor wishes to invest his wealth in certain risky assets, each of which has a constant scale of return that is a random variable. The objective of the investor is to maximize his expected utility of wealth subject to his budget constraint. The problem can be solved by a two-stage procedure, and the first stage solves a fractional program:

%(3) {\bf The single-period optimal portfolio selection problem}. The single-period optimal portfolio selection problem can be described as follows: An investor aims to allocate his wealth among various risky assets, each characterized by a consistent rate of return that is a stochastic variable. The investor's goal is to maximize the expected utility of their wealth while adhering to a budget constraint. The problem can be solved by a two-stage procedure, and the first stage solves a fractional program:
\item[(c)] {\bf The single-period optimal portfolio selection problem}. The single-period optimal portfolio selection problem \cite{pang1980parametric} can be solved by a two-stage procedure, and the first stage amounts to solving the following problem:

\begin{equation}
	\begin{aligned}
		&\min_{\h x\in S}\ \frac{\h x^{\top}V\h x}{{\mu}^{\top}\h x}\\
		& {\cal S} = \{{\h x}  : \h e^{\top}\h x = 1,\h 0 \leqslant \h x \leqslant \h d\}
	\end{aligned}\label{OPS}	
\end{equation}
where ${\h e}=(1,\ldots,1)^\top \in \mathbb{R}^n$, $\h d \in\mathbb{R}^{n}, {\mu}\in\mathbb{R}^{n}$ with ${\mu}^{\top}\h x > 0$ for all $\h x\in \mathcal{S}$. The matrix $V\in\mathbb{R}^{n\times n}$ is positive definite.
If we set  $h(\h x) = \h x^{\top}V\h x,\ g(\h x) = 0,\ f(\h x) = \h x, \ K = {\mu}^{\top},\ {\cal S} = \{{\h x} : \h e^{\top}\h x = 1, \h 0 \leqslant \h x \leqslant \h d\}$, (\ref{OPS}) can also be encompassed by the model (\ref{PForm}).

\item[(d)] {\bf Standard Sharpe ratio minimization problem}. The standard Sharpe ratio optimization problem \cite{CHZ11} commonly encountered in finance can be formulated as follows:
\begin{equation}
\begin{aligned}
	\min_{\h x\in S}&\ \frac{r - \h a^{\top}\h x}{(\h x^{\top} C\h x)^{1/2}}\\
	\cal{S} = &\{\h x :\ \h e^{\top}\h x = 1, \h x\geqslant \h 0\},
\end{aligned}	\label{SSR}	
\end{equation}
where ${\h e}=(1,\ldots,1) \in \mathbb{R}^n$, $({\h a}, r)\in{\mathbb R}^n\times {\mathbb R}$, are such that $r-{\h a}^\top {\h x} \ge 0$ for all $\h x \in {\cal{S}}$, and $C$ is positive definite matrix.
Problem (\ref{SSR}) can be written as a special case of (\ref{PForm}) if we take $h(\h x) = r - \h a^{\top}\h x, \ g(\h x) = 0, \ f(\h x) = ||\h x||,\ {\cal S} = \{{\h x} : \h e^{\top}\h x = 1, \h x \geqslant \h 0\}$, and $K$ is the Cholesky decomposition matrix of $C$ such that
$C=K^\top K$.

\end{itemize}
A conventional method for nonlinear fractional programming is Dinkelbach's method \cite{dinkelbach1967nonlinear, IB83}. In each iteration of Dinkelbach's method, the problem reduces to solving a difference-of-convex (DC) programming subproblem,

\begin{eqnarray}\label{DCi}
\min_{\h x} \left[g(\h x) + h(\h x) - \theta_k f(K{\h x})\right],
\end{eqnarray}
where \(\theta_k\) is updated at each iteration via:
\[
\theta_k := \frac{g(\h x^{k+1}) + h(\h x^{k+1})}{f(K{\h x}^{k+1})}.
\]

\noindent Solving the problem in each iteration of (\ref{DCi}) is challenging. In \cite{BC17}, Bot and Csetnek studied the single-ratio fractional programming problem with a convex, lower semicontinuous numerator, and a smooth concave or convex denominator. They proposed a proximal-gradient algorithm, which converges to the global optimum for a concave denominator and to critical points for a convex denominator under the Kurdyka-\L ojasiewicz (KL) property.

Building on this, Bot, Dao, and Li \cite{BDL} extended the approach by considering a composite numerator (nonsmooth, nonconvex combined with smooth, convex) and a weakly convex denominator. They introduced an extrapolated proximal subgradient algorithm (e-PSG). Their
 extrapolation parameter is sufficiently general to include those commonly used in the FISTA algorithms \cite{beck2017first} and established sequential convergence under a KL-satisfying merit function.

More recently, \cite{LSZ} and \cite{ZLSIAM} further advanced the problem. In \cite{ZLSIAM}, they proposed a proximal gradient-subgradient method with nonmonotone line search for the single-ratio fractional programming problem where the
numerator consists of a proper lower semicontinuous function and a smooth, nonconvex function,
 with a convex denominator. In \cite{LSZ}, a method with backtracked extrapolation (PGSA\_BE) was developed for solving (\ref{PForm}) with $K=I$. Additionally, \cite{BDL23} introduced an inertial proximal block coordinate method for nonsmooth sum-of-ratios problems. Convergence to critical points was established in all these approaches under functions satisfying the KL property.

The approaches mentioned above focus on solving the fractional programming problem that doesn't involve {\it a linear operator composed of nonsmooth convex} functions in the denominator. With one exception, Bot, Li, and Tao \cite{boct2023full}
considered the single-ratio fractional programming problem where
 the sum of a convex, possibly nonsmooth function composed with a linear operator and a differentiable, possibly nonconvex function in the numerator and a convex, possibly nonsmooth function composed with a linear operator in the denominator.
 In \cite{boct2023full}, they proposed a single-loop, fully split proximal subgradient algorithm with an extrapolated step, which uses a backtracking technique to ensure the involved fractional merit function's nonnegativity automatically. \cite{boct2023full} proved subsequential convergence toward an approximate lifted stationary point and established global
  convergence under the KL property. Furthermore, \cite{boct2023full} explained the rationale behind striving for an approximate lifted stationary point by constructing a series of counterexamples to demonstrate that seeking exact stationary solutions might lead to divergence.
  For extensive discussions on fractional programming,  we refer the reader to the monographs \cite{CuiPang} and \cite{SS95}.

The theoretical results in \cite{boct2023full} are interesting, and they are indeed tight in the sense that solving the composition of a double linear function with a nonsmooth convex function, both in the denominator and numerator, using a single-loop full-splitting algorithm with a global convergence guarantee under the KL assumption can only converge to an approximate lifted stationary point. Then, for the problem (\ref{PForm})  considered in this paper, where the linear operator appears only in the denominator of the objective function, a natural question arises:

\textit{ Can we develop a single-loop full-splitting algorithm for (\ref{PForm}) that converges to an exact lifted stationary point?}

In this paper, we answer this question affirmatively. Inspired by \cite{banert2019general,boct2023full}, we propose a single-loop fully splitting proximal subgradient algorithm. Additionally, we introduce a relaxation step in the algorithm to accelerate its convergence. We prove the subsequential convergence of the proposed algorithm to an exact lifted stationary point and establish its global convergence under a suitable merit function with the KL property.
Furthermore, we propose a practical version of our algorithm by employing a nonmonotone line search strategy. This version allows for larger step sizes and reduces the dependence on the unknown parameter $L_{\nabla h}$, thereby improving convergence speed.

The rest of the paper is organized as follows: In Section \ref{pre}, we provide fundamental definitions and preliminary results,
 introduce different stationary point notions, and state the necessary assumptions. Section \ref{newalgo} presents a single-loop fully-splitting proximal subgradient algorithm with a relaxation step  and establishes
 the subsequential convergence analysis. In Section \ref{sec6}, we prove the global convergence of the proposed algorithm by assuming
 some merit functions satisfy the KL property. In Section \ref{num}, we present the practical version of our algorithm and conduct extensive numerical experiments on three specific application problems to illustrate its efficiency. Section \ref{con} contains our conclusions.

\section{Preliminaries}\label{pre}
Let bold letters denote vectors, e.g., ${\bm x} \in \mathbb{R}^{n}$. Given two vectors ${\bm x}, {\bm y} \in \mathbb{R}^{n}$, $\langle {\bm x}, {\bm y} \rangle$ denotes their standard inner product: $\sum_{i=1}^n x_i y_i$. The notation $\|\bm{x}\|_p$ refers to the $p$-norm, defined as
$
\|\bm{x}\|_p = \left(\sum_{i=1}^n |x_i|^p \right)^{1/p}
$
for $0 < p < \infty$. The subscript $p$ in $\|\cdot\|_p$ is omitted when $p=2$. Define $[n]:=\{1, 2, \ldots, n\}$.
For a fixed $\kappa\in[n]$, $||{\h x}||_{(\kappa)}$ denotes the largest-$\kappa$ norm, i.e., the sum of the absolute values of the largest $\kappa$ elements. Let $r > 0$, $B({\h x}, r)$ denotes the ball centering at ${\h x}$ with radius $r$. The notations of ${\text{ri}}(\cdot)$ and ${\text{int}}(\cdot)$ denote the relative interior and the interior of a set.
%For any matrix $A$, we use $||A||_{1}$ to denote the sum of the absolute values of all elements in $A$ and $||A||_{\text{F}}$ be the Frobenius norm.

Define the extended real line ${\overline{\mathbb R}}={\mathbb R}\cup\{+\infty\}$.
For a function $f:{\mathbb R}^n\rightarrow{\overline{\mathbb R}}$, the domain of $f$ is defined by $\text{dom}f = \{\h x :\ f(\h x) < +\infty\}$.
We call a function $f$ is ${\cal C}^1$ over an open set, which means that it is differentiable with Lipschitz continuous gradient over an open set.
Given a closed convex set ${\cal C}$, we use $\iota_{\cal C}(\cdot)$ to denote the indicator function of ${\cal C}$.
For a function $f:{\mathbb R}^n\rightarrow{\overline{\mathbb R}}$
and finite at $\overline{\h x}$. The set

\begin{equation*}
		{\hat\partial} f({\overline{\h x}})=\left\{{\h v}\ : {\mathop{\varliminf}\limits_{\h x\to{\overline{\h x}}\;{\h x}\neq {\overline{\h x}}}} \frac{f(\h x)-f(\overline{\h x})-\langle {\h v},{\h x}-{\overline{\h x}}\rangle}{\|{\h x}-\overline{\h x}\|}\ge 0\right\},
\end{equation*}
is called a Fr\'{e}chet subdifferential of $f$ at $\overline{\h x}$. Its elements are called Fr\'{e}chet subgradients.
%If ${\hat\partial} f({\overline{\h x}})\neq\emptyset$, then $f$ is  lower semicontinuous at $\overline{\h x}$ \cite{KAY}.
The limiting subdifferential $\partial f(\overline{\h x})$   is defined as
\begin{equation*}
		\partial f({\overline{\h x}}):=\left\{{\h v}\; : \;\exists\; {\h x}^k\rightarrow {\overline{\h x}},\;f({\h x}^k)\rightarrow f(\overline{\h x}),
 {\h v}^k\in{\hat\partial} f({\h x}^k), {\h v}^k\rightarrow{\h v} \right\}.
		\end{equation*}
Obviously, $\hat\partial f({\overline{\h x}})\subseteq \partial f({\overline{\h x}})$ for all ${\overline{\h x}}\in\mathbb R^n$.
$\hat\partial f({\overline{\h x}})$ is closed and convex set while  $\partial f({\overline{\h x}})$ is closed \cite[Theorem 8.6]{RockWets}.
If $f$ is convex function, $\hat\partial f({\overline{\h x}})=\partial f({\overline{\h x}})$.

Let $f: {\mathbb R}^n\rightarrow (-\infty,+\infty]$. Its convex  conjugate of $f^*$ is defined by
$$f^*({\h z})=\sup_{\h x} \{\langle{\h x},{\h z} \rangle -f({\h z}) \}.$$

\noindent Given a proper lower semicontinuous function  $f$, we define  its  proximal mapping  \cite[Definition 1.22]{RockWets}:
 $$\prox_{f,\;\kappa}({\h v}) = \arg\min_{\h x} \Big\{ f({\h x}) + \frac{1}{2\kappa}\|{\h x}-{\h v}\|^2\Big\}.$$
  When $f({\h x})=\|{\h x}\|_1$, its proximal operator has a closed-form solution and characterized by
\begin{eqnarray*}\text{\rm{Shrink}}({\h v},\kappa)= \max(|{\h v}|-\kappa,0)\odot{\text{sign}}({\h v}),\end{eqnarray*}
where $\odot$ means componentwise multiplication.
Given an image ${\h u}\in {\mathbb R}^{n\times n}$, we define a discrete gradient operator:
\begin{eqnarray*}({\nabla {\h u}})^\top =\left((\nabla_{\h x}{\h u})^\top,(\nabla_{\h y}{\h u})^\top\right),\end{eqnarray*}
where ${\nabla_{\h x}}$ and ${\nabla_{\h y}}$ are the forward  horizontal
and vertical difference operators and
${\nabla}_{\h x},{\nabla}_{\h y}:{\mathbb R}^{n\times n}\mapsto {\mathbb R}^{n\times n}$ and
``$\mapsto$" represents a mapping. %Obviously, $\nabla: {\mathbb R}^{n\times n}\mapsto {\mathbb R}^{{2n^2}\times n^2}$.
Next, we review the Kurdyka-\L ojasiewicz (KL) property  \cite{BDL07}, which is widely used in convergence analysis.
We refer the readers to \cite{ABRS,ABS} for more discussions.
\begin{definition} \label{def:KL}
A proper lower semicontinuous function
$h:\!\mathbb R^n\rightarrow(-\infty,+\infty]$
	satisfies the Kurdyka-\L ojasiewicz (KL) property at a point ${\hat {\h x}}\in {\text{\rm dom}}\partial h$ if there exist a constant $a\in(0,\infty]$, a neighborhood
	$U$ of ${\hat {\h x}},$ and a continuous concave function $\phi:\;[0,a)\rightarrow[0,\infty)$ with
	$\phi(0)=0$ such that
	\begin{itemize}
		\item[(i)] $\phi$ is continuously differentiable on $(0,a)$ with $\phi'>0$ on $(0,a)$;
		\item[(ii)] for every $\h x\in U$ with $h({\hat {\h x}})< h(\h x)< h(\hat {\h x}) +a$, it holds that
		\begin{eqnarray*}\phi'(h(\h x)-h(\hat {\h x})){\text{\rm{dist}}}(\h 0,\partial h({\h x}))\ge 1. \end{eqnarray*}
	\end{itemize}
\end{definition}

\begin{definition}\label{calm} \cite{RockWets}
The function $f: {\mathbb R}^n \rightarrow {\overline{\mathbb R}}$ is said to satisfy the {\it calm} condition at ${\h x}\in {\text{\rm dom}}(f)$,
 if there exist $\varrho>0$ and neighborhood $B({\h x},\varepsilon)$ of ${\h x}$ , such that
$$ | f({\h x})-f({\h y})|\le \varrho \|{\h x}-{\h y}\|$$
for any ${\h y}\in B({\h x},\varepsilon)$.
%We say $f$ satisfies the calm condition on ${\cal S}$ if $f$ satisfies the calm condition at any point in ${\cal S}$ relative to ${\cal S}$.
\end{definition}

%\noindent Next, we review several concepts from variation analysis.
%\begin{Def}\label{defprox} ({\bf Prox-regularity}) \cite{PR96}
%A function $f$ is prox-regular at a point ${\overline{\h x}}$ for a subgradient ${\overline {\h v}}\in \partial f({\overline{\h x}})$ if
%$f$ is finite at ${\overline{\h x}}$, locally lower semi-continuous around ${\overline{\h x}}$, and
%there exists $\rho>0$ such that
%$$ f({\h x}')\ge f({\h x})+ \langle{\h v},{\h x}'-{\h x} \rangle-\frac{\rho}{2}\|{\h x}'-{\h x}\|^2$$
%whenever ${\h x}$ and ${\h x}'$ are near ${\overline{\h x}}$ with $f({\h x})$ near $f({\overline{\h x}})$ and ${\h v}\in \partial f({\h x})$ is
%near ${\overline{\h v}}$.
%Furthermore, $f$ is prox-regular at ${\overline{\h x}}$ if it is prox-regular at ${\overline{\h x}}$ for every  ${\h v}\in\partial f({\overline{\h x}})$.
%\end{Def}

\noindent The following lemma provides a calculus rule of the Fr\'{e}chet subdifferential of the ratios of two functions.
\begin{lemma}\cite{boct2023full}\label{ratioC2}
Let $O \subseteq \R^n$ be an open set, and $f_1 : O \rightarrow \overline{\mathbb \R}$ and $f_2 : O \rightarrow \mathbb \R$ be two functions which are finite at ${\h x} \in O$ with $f_2({\h x})>0$. Suppose that $f_1$ is continuous at ${\h x}$ relative to
${\text{\rm dom}}f_1$, that $f_2$  is calm at ${\h x}$, and denote $\alpha_i:=f_i({\h x})$, $i=1,2$.

\noindent (i) Then
\begin{eqnarray*}\label{Lem3:eq1} {\hat\partial} \left(\frac{f_1}{f_2}\right)({{\h x}}) =\frac{{\hat\partial} (\alpha_2 f_1-\alpha_1 f_2)({{\h x}})}{f_2({{\h x}})^2}.\end{eqnarray*}
(ii) If, in addition, $f_2$ is convex  and $\alpha_1\ge0$, then
\begin{eqnarray*}\label{Lem3:eq2}{\hat\partial} \left(\frac{f_1}{f_2}\right)({{\h x}}) \subseteq \frac{{\hat\partial} (\alpha_2 f_1)({{\h x}})-\alpha_1{\hat\partial} f_2({{\h x}})}{f_2({{\h x}})^2}.\end{eqnarray*}
\end{lemma}

%  Let  $\phi_1,\phi_2$ be proper lower semicontinous functions.
%  $\phi_2^*$ denotes the conjugate function of $\phi_2$.
%Given $\nu>0$,
%  define $\eta({\h x},{\h y})$ by
%\begin{eqnarray*}\eta({\h x},{\h y})=\left\{\begin{array}{ll}
%\frac{\phi_1({\h x})}{\langle {\h x},{\h y}\rangle-\phi_2^*({\h y})}&\;\;\mbox{if}\;({\h x},{\h y })\in {\text{\rm dom}}(\phi_1)\times {\text{\rm dom}}(\phi_2^*)\\
%& \;\;\mbox{and}\;\langle{\h x},{\h y}\rangle-\phi_2^*({\h y})>0,\\[0.2cm]
% +\infty & \;\;\mbox{otherwise}.\end{array}\right.\end{eqnarray*}

\begin{lemma} \label{sharpcalm}
Let ${\h x}\in{\text{\rm dom}}\phi_1$,
${\h y}\in {\text{\rm{dom}}}(\phi_2^*)$ with $\phi_1({\h x})>0$ and
$\langle{\h x},{\h y} \rangle -\phi^*_2({\h y})> \nu>0$.
Define
$$\eta({\h x},{\h y}) =\frac{\phi_1({\h x})}{ \langle{\h x},{\h y} \rangle -\phi^*_2({\h y})}.$$
Assume that $\phi_1$ is continuous at ${\h x}$ relative to ${\text{\rm dom}}(\phi_1)$ and
$\phi_2^*$ satisfies  calm at ${\h y}$.
Then, there exists $B({\h x},\varepsilon)\times B({\h y},{\varepsilon})$
 such that for any $({\h u},{\h v})\in B({\h x},\varepsilon)\times B({\h y},{\varepsilon})$,
${\h v}\in{\text{\rm dom}}(\phi_2^*)$ and $\phi_2^*$ is  calm on $B({\h y},{\varepsilon})$
and $\langle {\h u},{\h v} \rangle-\phi_2^*({\h v})\ge \nu>0$.
Denote
${\hat{\partial}}\eta({\h x},{\h y}) = {\hat{\partial}}_{\h x}\eta({\h x},{\h y}) \times {\hat{\partial}}_{\h y}\eta({\h x},{\h y}).$
 Then,
\begin{eqnarray}\label{etax}&&{\hat{\partial}}_{\h x}\eta({\h x},{\h y}) = \frac{(\langle{\h x},{\h y} \rangle -\phi^*_2({\h y}))\hat{\partial}\phi_1({\h x})-\phi_1({\h x}){\h y}}{(\langle{\h x},{\h y} \rangle -\phi^*_2({\h y}))^2}, \\[0.2cm]
\label{etay} &&{\hat{\partial}}_{\h y}\eta({\h x},{\h y})=\frac{\phi_1({\h x})({\hat{\partial}\phi_2^*}({\h y})-{\h x})}{(\langle{\h x},{\h y} \rangle -\phi^*_2({\h y}))^2}.\end{eqnarray}
 \end{lemma}
 \begin{proof}
 Define the function
 \begin{eqnarray*}\eta_1({\h x},{\h y})=\phi_1({\h x}),\;\; \eta_2({\h x},{\h y})= \langle{\h x},{\h y} \rangle-\phi_2^*({\h y}).\end{eqnarray*}
 By invoking the assertion (i) of Lemma \ref{ratioC2}, we have
 \begin{eqnarray*}{\hat\partial}\eta({\h x},{\h y}) =\frac{{\hat\partial}(\alpha_2 \eta_1-\alpha_1\eta_2)({\h x},{\h y})}{\eta_2({\h x},{\h y})^2}.
 \end{eqnarray*}
 With further calculations, we have the desired assertions (\ref{etax}) and (\ref{etay}).
 \end{proof}

\begin{remark}\label{remk25}
We note that Lemma \ref{sharpcalm} differs slightly from \cite[Proposition 2.3]{LSZ}, where ``$\phi_2^*$ satisfies calmness at ${\h y}$ relative to $\dom(\phi_2)$". Indeed,
 assuming for $\phi_2^*$ calmness relative to its domain will be dangerous since we want to have a neighborhood of the point(s) at which we calculate the subdifferential where $\phi_2^*$  is finite.
\end{remark}
%\begin{remark} \label{remk26} If $\phi_2^*$ is an indicator function of a closed convex set,
%the assumption ``$\langle {\h u},{\h v} \rangle-\phi_2^*({\h v})\ge \nu>0$"
%can guarantee $\h y$ cannot lie on the boundary of $\dom\phi_2^*$, i.e., ${\h y}\in \text{int}(\dom\phi_2^*)$.
%\end{remark}
%\section{Stationary points and assumptions of fractional programs}
Next, we introduces the following assumptions for the problem  (\ref{PForm}).

\begin{ass}\label{ass1}
Throughout this paper, we assume that:
\begin{itemize}
    \item[(a)] ${\cal S} \subseteq {\mathbb R}^n$ is a nonempty, convex, and compact set;
    \item[(b)] $g$ is a proper, convex, and lower semicontinuous function;
    \item[(c)] $h$ is differentiable with a Lipschitz continuous gradient over an open set containing the compact set ${\cal S}$, with Lipschitz constant $L_{\nabla h}$;
    \item[(d)] $f$ is a proper, convex, and lower semicontinuous function such that $K(\mathcal{S}) \subseteq {\rm int}({\rm dom} f)$ and $f(K{\h x}) > 0$ for all ${\h x} \in {\cal S}$;
    \item[(e)] ${\cal S} \cap ({\text{\rm dom}}g) \neq \emptyset$ and $\alpha := \inf_{\h x \in {\cal S}} \{ g({\h x}) + h({\h x})\} \ge 0$.
\end{itemize}
\end{ass}

Next, we present an additional assumption to ensure that the calculus rules of the subdifferential hold,
i.e., $\partial (\iota_{S}(\cdot)+g(\cdot))=\partial \iota_{S}(\cdot)+\partial g(\cdot)$.

\begin{ass}\label{ass2}
    $\text{ri}(\mathcal{S}) \cap \text{ri}(\text{dom} g) \neq \emptyset$ or $\mathcal{S} \cap \text{ri}(\text{dom} g) \neq \emptyset$ if ${\cal S}$ is polyhedral.
\end{ass}

Next, we introduce two versions of stationary points for fractional programs and explore their relationships.

\begin{definition} \cite{BDL} For problem (\ref{PForm}), we define that ${\overline{\h x}}$ is:
\begin{itemize}
    \item[(i)] a Fr\'{e}chet stationary point if
    $
    0 \in \hat{\partial} \displaystyle{\left( \frac{g + h + \iota_{\cal S}}{f \circ K} \right)({\overline{\h x}})};
    $
    \item[(ii)] a (limiting) lifted stationary point if
    $$
    0 \in \left( \partial g({\overline{\h x}}) + \nabla h(\overline{\h x}) + \partial \iota_{\cal S}({\overline{\h x}}) \right) f(K \overline{\h x}) - (g(\overline{\h x}) + h(\overline{\h x})) K^* \partial f(K {\overline{\h x}}).
    $$
\end{itemize}
\end{definition}

\noindent Any local minimizer \(\overline{\h{x}} \in \mathbb{R}^n\) of (\ref{PForm}) is clearly a Fr\'{e}chet stationary point under Assumption \ref{ass1}. Conversely, if \(\overline{\h{x}} \in \mathbb{R}^n\) is a Fr\'{e}chet stationary point of (\ref{PForm}) such that either \(\overline{\h{x}} \in \text{ri}(\mathcal{S}) \cap \text{ri}(\text{dom}\,g)\), or \(\mathcal{S}\) is polyhedral and \(\overline{\h{x}} \in \mathcal{S} \cap \text{ri}(\text{dom}\,g)\), then, by Lemma \ref{ratioC2}, \(\overline{\h{x}}\) is also a limiting lifted stationary point of (\ref{PForm}). However, a limiting lifted stationary point is not necessarily a Fr\'{e}chet stationary point (see, e.g., \cite[Example 3.4]{boct2023full}).
In the numerical section, we will discuss quantifying the violation of the conditions defining a (limiting) lifted stationary point.

\begin{remark} Obviously, the motivation examples (a)-(d) all satisfy Assumptions \ref{ass1} and \ref{ass2}.
\end{remark}

\section{Proposed algorithm}\label{newalgo}
\subsection{Fully-splitting proximal subgradient}
In this section, we introduce a single-loop fully-splitting proximal subgradient algorithm with a relaxation step (abbreviated as FPSA) for solving problem~(\ref{PForm}). Inspired by the effectiveness of relaxation techniques in both convex and nonconvex optimization, we incorporate a relaxation step in FPSA to accelerate its convergence;  employed a similar approach \cite{boct2023full}. When the proximal mapping of \( g \) admits a closed-form solution, FPSA becomes a single-loop, fully splitting algorithm, which we refer to as a proximal-friendly algorithm.

The step size \( \delta \) is allowed to vary across iterations, satisfying
\[
0 < \underline{\delta} \le \inf_{k \in \mathbb{N}} \delta_k \le \delta_k \le \sup_{k \in \mathbb{N}} \delta_k \le \overline{\delta}<\frac{1}{L_{\nabla h}}.
\]
Similarly, the relaxation parameter \( \sigma \) can also depend on the iteration index, satisfying
\[
0 < \underline{\sigma} \le \inf_{k \in \mathbb{N}} \sigma_k \le \sigma_k \le \sup_{k \in \mathbb{N}} \sigma_k \le \overline{\sigma}<2.
\]
These flexible choices of step size may influence the algorithm behavior in some ways, and the coming convergence analysis can easily extend to this more general step size setting.
To illustrate the core insights of the proposed algorithm, we adopt constant step sizes throughout our presentation and analysis:

%{\bf  (\ref{SplitI}): Full-splitting proximal subgradient algorithm with extrapolated step (Fs-PSA$_{\text e}$).}
\begin{subequations} \label{SplitI}
	\begin{numcases}{\hbox{\quad}}
{\h y}^{k+1}\in\partial f(K{\h x}^k)\label{ysub}\;\\[0.1cm]
\label{xsub} \h x^{k+1}\in\arg\min_{{\h x}\in{\cal S}} \left[ g({\h x})+\frac{1}{2\delta}\left\|{\h x}-{\h u}^k+\delta\nabla h({\h x}^k)-\theta_k\delta K^\top {\h y}^{k+1}\right\|^2\right]\;\;\\
\label{usub}{\h u}^{k+1}=(1-\sigma){\h u}^k+\sigma{\h x}^{k+1}\;\\[0.1cm]
%\label{wsub}{\h v}^{k+1}={\h v}^k+\beta({\h v}^k-{\h z}^{k+1}),\\
\theta_{k+1}=\vartheta({\h x}^{k+1},{\h u}^{k+1};\delta)\label{theta}
%\label{lamsub}{\bm \gamma}^{k+1}={\bm \gamma}^k+\rho(A{\h x}^{k+1} -{{\h z}}^{k+1}).
	\end{numcases}
\end{subequations}
%\begin{algorithm}[t]
%	\caption{Full-splitting proximal subgradient algorithm with extrapolated step (Fs-PSA$_{\text e}$)}
%	\label{subx:admm:Bes}
%	\begin{algorithmic}[1]
%		\Require{$\beta >0$, $\iota>0$, $0<q<1$ and $\delta_0,\theta_0,\sigma_0>0$; MaxItr.}
%  \For k=1:MaxItr
%   \State{${\h y}^{k+1}\in\partial f(K{\h x}^k).$}
%   \State{$\h x^{k+1}\in{\text{Proj}}_{\cal S}\left({\h u}^k+\frac{\theta_k}{\delta}K^\top {\h y}^{k+1}-\frac{1}{\delta}\nabla h({\h x}^k)-\frac{1}{\delta}A^\top {\h z}^k\right)$.}
%   \State{${\h{z}}^{k+1}=\arg\min_{\h z}\left( g^*({\h z})-\langle A{\h x}^{k+1}, {\h z}\rangle+\frac{\sigma}{2}\|\h z\|^2\right)$.}
%   \State{Choosing $\sigma_{k+1}=q*\sigma$  such that }
%   \State{$\theta^{k+1}=\vartheta({\h x}^{k+1},{\h z}^{k+1},{\h u}^{k+1};\delta,{\sigma},\sigma_{k+1})>0$.}
%   \State{Update $\delta_{k+1} = 2\iota +L_{\nabla h}+ \frac{\|A\|^2}{\sigma_{k}}+\frac{\|A\|^2}{\sigma_{k+1}^2}$.}
%   \State{${\h u}^{k+1}={\h u}^k-\beta({\h u}^k-{\h x}^{k+1})$.}
%   \EndFor
%	\State{{\bf Output} ${\overline {\h x}}$.}
%	\end{algorithmic}
%\end{algorithm}
and
\begin{eqnarray}\label{Lfun}
&&\varphi({\h x},{\h u};\delta)=g({\h x})+h({\h x})+\iota_{\cal S}({\h x})+\frac{1}{2\delta}\|{\h x}-{\h u}\|^2,\\
&& \vartheta({\h x},{\h u};\delta) = \frac{\varphi({\h x},{\h u};\delta)}{f({K}{\h x})}.
\end{eqnarray}
%Although the parameter $\beta$ in (\ref{usub}) can change from iteration to iteration, we fix it here for simplicity.
%We specify
%the parameters of $\delta,\sigma$ as follows:
%\begin{center}\fbox{
% \begin{minipage}{12.8cm}
%
%\noindent {\bf Requirement of $\sigma$ and $\delta$}.
%\begin{subequations}
%\begin{eqnarray}\label{gam} \sigma_{k+1}\le \sigma\;\mbox{and}\; \sigma\ge {\overline{\sigma}}>0,\end{eqnarray}
%\begin{eqnarray}\label{del} \delta=2\iota+L_{\nabla h}+\frac{\|A\|^2}{\sigma_{k-1}}+\frac{\|A\|^2}{\sigma^2}, \end{eqnarray}
%\end{subequations}
%where $\iota>0$.
%\end{minipage}
% \smallskip
%} \end{center}

\subsection{Subsequential convergence}

Define ${\bm W}^k:=({\h x}^k,{\h y}^k,{\h u}^k)$ and denote
\begin{equation*}\label{cofsharp}
 \varrho_{1}=\frac{1-\delta L_{\nabla h}}{2\delta},\;\;\varrho_{2}={\frac{(1-(1-\sigma)^2)}{2\sigma^2 \delta}}.
\end{equation*}
We first establish the following theorem to show the descent property of the combination of $\varphi$
with an extra term scaling with $\theta_k$ involved the denominator of (\ref{PForm}) or its conjugate.

\begin{theorem}\label{PriTheo}
Let the sequence $\{{\bm W}^k\}$
be generated by (\ref{SplitI}).
 Suppose Assumption \ref{ass1} holds.
Then,
for any $k$, it holds that
 \begin{eqnarray*}
 &&\varphi({\h x}^{k+1},{\h u}^{k+1};\delta)+\theta_k f({K}{\h x}^k)-\theta_k\left( \langle {K}{\h x}^{k+1},{\h y}^{k+1} \rangle-f^*({\h y}^{k+1})\right)\nn\\
&&\le\varphi({\h x}^{k},{\h u}^{k};\delta)
-\varrho_{1}\|{\h x}^k-{\h x}^{k+1}\|^2-\varrho_{2}\|{\h u}^k-{\h u}^{k+1}\|^2.
\end{eqnarray*}

\end{theorem}

\begin{proof}
 Since $\nabla h$ is Lipschitz continuous with constant $L_{\nabla h}$, it holds that
\begin{eqnarray}\label{desx:key1}h({\h x}^{k+1})-h({\h x}^k)\le \langle {\h x}^{k+1}-{\h x}^k,
\nabla h({\h x}^k)\rangle +\frac{L_{\nabla h}}{2}\|{\h x}^{k+1}- {\h x}^k\|^2. \end{eqnarray}
Noting that  the objective of (\ref{xsub}) is $1/\delta$-strongly convex, we have
\begin{eqnarray}\label{desx:key2} &&g({\h x}^{k+1})+\langle {\h x}^{k+1}-{\h x}^k,\nabla h({\h x}^k)\rangle-\theta_k\langle {K}^\top{\h y}^{k+1},{\h x}^{k+1} \rangle
+\frac{1}{2\delta}\|{\h x}^{k+1}-{\h u}^k\|^2\nn\\
&&\le {g({\h x}^{k})}-\theta_k\langle {K}^\top
{\h y}^{k+1},{\h x}^k \rangle+\frac{1}{2\delta}\|{\h x}^k-{\h u}^k\|^2-\frac{1}{2\delta}\|{\h x}^k-{\h x}^{k+1}\|^2.
 \end{eqnarray}
 \noindent By combining the above inequality with the fact
$$\langle {\h y}^{k+1},{K}({\h x}^k-{\h x}^{k+1})\rangle =f({K}{\h x}^k)-\left(\langle {K}{\h x}^{k+1},{\h y}^{k+1} \rangle-f^*({\h y}^{k+1})\right),$$
we have that
\begin{eqnarray*} &&g({\h x}^{k+1})+\langle {\h x}^{k+1}-{\h x}^k,\nabla h({\h x}^k)\rangle+\theta_k\left[f({K}{\h x}^k)- (\langle {K}{\h x}^{k+1},{\h y}^{k+1} \rangle -f^*({\h y}^{k+1}) )\right]
\nn\\ &&+\frac{1}{2\delta}\|{\h x}^{k+1}-{\h u}^k\|^2
\le g({\h x}^{k})+\frac{1}{2\delta}\|{\h x}^k-{\h u}^k\|^2-\frac{1}{2\delta}\|{\h x}^k-{\h x}^{k+1}\|^2.
 \end{eqnarray*}

\noindent  Substituting (\ref{desx:key1}) into (\ref{desx:key2}), we have
 \begin{eqnarray*} &&g({\h x}^{k+1}) +h({\h x}^{k+1})-h({\h x}^k)+\theta_k\left[f({K}{\h x}^k)- (\langle {K}{\h x}^{k+1},{\h y}^{k+1} \rangle -f^*({\h y}^{k+1}) )\right]
\nn\\ &&+\frac{1}{2\delta}\|{\h x}^{k+1}-{\h u}^k\|^2
\le g({\h x}^{k})+\frac{1}{2\delta}\|{\h x}^k-{\h u}^k\|^2-\frac{1/\delta-L_{\nabla h}}{2}\|{\h x}^k-{\h x}^{k+1}\|^2.
 \end{eqnarray*}

\noindent It further implies that
\begin{eqnarray}\label{psides1} &&\varphi({\h x}^{k+1},{\h u}^k;\delta) +\theta_k\left[f({K}{\h x}^k)- (\langle {K}{\h x}^{k+1},{\h y}^{k+1} \rangle -f^*({\h y}^{k+1}) )\right]\nn\\
 &&\le \varphi({\h x}^{k},{\h u}^k;\delta) -\frac{1}{2}(1/\delta-L_{\nabla h})\|{\h x}^{k+1}-{\h x}^k\|^2.\end{eqnarray}

\noindent Next,
\begin{eqnarray*}
 &&\frac{1}{2\delta}\|{\h u}^{k+1}-{\h x}^{k+1}\|^2=\frac{1}{2\delta}\|{\h u}^k-\sigma({\h u}^k-{\h x}^{k+1})-{\h x}^{k+1}\|^2\nn\\
 %&=&\frac{1}{2\kappa}\|{\h u}^k-{\h y}^{k+1}\|^2-\frac{1}{2\kappa}(1-(1-\xi)^2)\|{\h u}^k-{\h y}^{k+1}\|^2\nn\\
 &&=\frac{1}{2\delta}\|{\h u}^k-{\h x}^{k+1}\|^2-\frac{(1-(1-\sigma)^2)}{2\sigma^2 \delta}\|{\h u}^k-{\h u}^{k+1}\|^2,
 \end{eqnarray*}
 where the last equality follows from (\ref{usub}).
 So, it leads to
 \begin{eqnarray}
\varphi({\h x}^{k+1},{\h u}^{k+1}; \delta)\ \le\ \varphi({\h x}^{k+1},{\h u}^{k};\delta)-\frac{(1-(1-\sigma)^2)}{2\sigma^2\delta}\|{\h u}^k-{\h u}^{k+1}\|^2.\label{phi_descent}
\end{eqnarray}
%Finally, the assertion follows directly by combining (\ref{psides1}) with the above inequality.\qed
Finally, the assertion follows directly by combining (\ref{psides1}) with the above inequality.
% \noindent Finally, the assertion follows directly by combining (\ref{psides1}) with the above inequality.\qed
\end{proof}

\begin{theorem}\label{DescentLemma}Let the sequence $\{{\bm W}^k\}$
be generated by (\ref{SplitI}). Suppose Assumption \ref{ass1} holds. Then,
  for any $k$,
 \begin{eqnarray*}&&\varphi({\h x}^{k+1},{\h u}^{k+1};\delta)-\theta_k f({K}{\h x}^{k+1})\le
-\varrho_1\|{\h x}^k-{\h x}^{k+1}\|^2-\varrho_2\|{\h u}^k-{\h u}^{k+1}\|^2.
 \end{eqnarray*}
\end{theorem}
\begin{proof}
 Since
$\langle {K}{\h x}^{k+1},{\h y}^{k+1} \rangle-f^*({\h y}^{k+1}) \le  f({K}{\h x}^{k+1})$,
and  $\theta_k \ge 0$,
 Theorem \ref{PriTheo} leads to
\begin{eqnarray*}
 &&\varphi({\h x}^{k+1},{\h u}^{k+1};\delta)+\theta_k f({K}{\h x}^k)-\theta_k f({K}{\h x}^{k+1})\nn\\
&&\le\varphi({\h x}^{k},{\h u}^{k};\delta)
-\varrho_1\|{\h x}^k-{\h x}^{k+1}\|^2-\varrho_2\|{\h u}^k-{\h u}^{k+1}\|^2.
 \end{eqnarray*}
By using
$\theta_k f({K}{\h x}^k)=\varphi({\h x}^{k},{\h u}^{k};\delta)$, the desired inequality follows.
\end{proof}

\begin{remark}\label{remk1}
	Since ${\rm int}({\text{\rm dom}}f)\neq \emptyset$ and
	${K}({\cal S})\subseteq {\rm int}(\text{dom} f)$, there exist two positive scalars $m$ and $M$ such that
	$$0<m<
	\inf\limits_{k\in{\mathbb N}}f({K}{\h x}^k)\le \sup\limits_{k\in{\mathbb N}}f({K}{\h x}^k)\le M.$$
\end{remark}

\begin{remark}({\it Boundedness of the sequence})
Since  $\{{\h x}^k\}$ is bounded,
and $int({\text{\rm dom}}f)\neq \emptyset$,  $\{{\h y}^k\}$ is bounded   \cite[Theorem 23.4]{Rock70}.
Hence, there exists a constant ${\cal M}$ such that $\|{\h x}^k\|\le {\cal M}$ for all $k$. Then
$\|{\h u}^k\|\le \|{\h u}^0\| + \frac{\sigma\mathcal{M}}{1-|\sigma-1|}$ when $0<\sigma<2$. %In addition, we see that $\{{\h z}^k\}$ is  bounded.
\end{remark}

\begin{theorem}\label{subsequentialtoStat}
Let the sequence $\{{\bm W}^k\}$ be generated by (\ref{SplitI}) and $\Omega$ represents
the set of accumulation points of $\{{\bm W}^k\}$.  Suppose that Assumptions \ref{ass1} and \ref{ass2} hold.
The parameters of  $\sigma$ and $\delta$ are defined in (\ref{SplitI}) satisfying $0<\sigma<2$ and $0<\delta<1/{L_{\nabla h}}$.
Then,
\begin{itemize}
\item[(i)]$\lim_{k\to\infty}\theta_k$ exists, denoted by ${\overline \theta}$, and $\overline{\theta}\ge 0$;
    \item[(ii)]
    $\vartheta({\h x},{\h u},\delta) \equiv{\overline \theta}$ on $\Omega$ where
 $\vartheta$ is defined in (\ref{Lfun});
\item[(iii)]
 Any accumulation point $\overline{\h x}$ of $\{{\h x}^k\}$ is a (limiting)  lifted stationary point (\ref{PForm}).
    Moreover, it holds that $\overline{\theta}=\frac{{g({\overline{\h x}})}+h(\overline{\h x})}{f({K}{\overline{\h x}})}$ for any accumulation point ${\overline{\h x}}$;
\item[(iv)]It holds that
\begin{eqnarray*}\label{eqde}\lim_{k\to+\infty}\frac{\langle {K}{\h x}^{k+1}, {\h y}^{k+1} \rangle -f^*({\h y}^{k+1})}{f({K}{\h x}^{k})}=1.\end{eqnarray*}

\end{itemize}

\end{theorem}
\begin{proof}
(i) It follows Theorem \ref{DescentLemma} that
\begin{eqnarray*}&&\theta_{k+1}\le \theta_k
-\varsigma_1\|{\h x}^k-{\h x}^{k+1}\|^2-\varsigma_2
\|{\h u}^k-{\h u}^{k+1}\|^2,\label{thetades} \end{eqnarray*}
where $\varsigma_i=\frac{c_i}{M}$ ($i=1,2$).
Consequently, the sequence of $\{\theta_k\}$
is nonincreasing. In addition,  $\theta_k\ge 0$.
Thus, $\lim_{k\to\infty}\theta_k={\overline{\theta}}$ exists. Obviously, $\overline{\theta}\ge 0$.
Moreover,
\begin{eqnarray*}\label{acc1} \|{\h x}^k-{\h x}^{k+1}\|\rightarrow 0,\;\|{\h u}^k-{\h u}^{k+1}\|\rightarrow 0,\end{eqnarray*}
as $k\to+\infty$.

\noindent (ii) Suppose that ${\hat {\bm W}}=({\hat{\h x}},{\hat {\h y}},{\hat{\h u}})\in \Omega$. Then, there exists a subsequence $\{{\bm W}^{k_j}\}$ such that
$\lim_{j\to\infty}\|{\bm W}^{k_j}-{\hat {\bm W}}\|=0.$
 Since  $\varphi$ is lower semicontinuous and
$f$ is continuous, $\vartheta$ is lower semicontinuous on ${\cal S}$.
Thus, we have that

\begin{eqnarray}\label{objpro}&&\vartheta({\hat{\h x}},{\hat{\h u}},\delta) \le\varliminf\limits_{j\to +\infty}\frac{\varphi({\h x}^{k_j},{\h u}^{k_j},\delta)}{f({K}{\h x}^{k_j})}=\lim_{j\to+\infty} \frac{\varphi({\h x}^{k_j},{\h u}^{k_j},\delta)}{f({K}{\h x}^{k_j})}={\overline\theta},\end{eqnarray}
 where the first equality is due to the lower semicontinuity of $\vartheta$.

On the other hand, by invoking (\ref{xsub}), we can derive that
\begin{eqnarray*}
	&& g(\h x^{k_{j}+1}) + \langle \nabla h(\h x^{k_{j}}), \h x^{k_{j}+1} - \hat{\h x}\rangle + \frac{1}{2\delta}||\h x^{k_{j}+1} - \h u^{k_{j}}||_{2}^{2}\notag \\
	&& \leqslant g(\hat{\h x}) + \theta_{k_{j}}\langle {K}^{\top}\h y^{k_{j}+1}, \h x^{k_{j}+1} - \hat{\h x}\rangle + \frac{1}{2\delta}||\hat{\h x} - \h u^{k_{j}}||_{2}^{2}.\label{g_sub_hat}
\end{eqnarray*}
Since $h$ is Lipschitz continuous with constant $L_{\nabla h}$, it holds that:
\begin{eqnarray*}
	h(\h x^{k_{j}+1}) \leqslant h(\hat{\h x}) + \langle \nabla h(\hat{\h x}), \h x^{k_{j}+1}-\hat{\h x}\rangle + \frac{L_{\nabla h}}{2}||\h x^{k_{j}+1} - \hat{\h x}||^{2}.\label{h_sub_hat}
\end{eqnarray*}
Adding the above two inequalities leads to
\begin{eqnarray}
	&& g(\h x^{k_{j}+1}) + h(\h x^{k_{j}+1}) + \frac{1}{2\delta}||\h x^{k_{j}+1} - \h u^{k_{j}}||_{2}^{2}\notag \\
	&& \leqslant g(\hat{\h x}) + h(\hat{\h x}) + \frac{1}{2\delta}||\hat{\h x} - \h u^{k_{j}}||_{2}^{2} + \frac{L_{\nabla h}}{2}||\h x^{k_{j}+1} - \hat{\h x}||^{2} \notag\\
	&& + \theta_{k_{j}}\langle {K}^{\top}\h y^{k_{j}+1}, \h x^{k_{j}+1} - \hat{\h x}\rangle - \langle \nabla h(\h x^{k_{j}}) - \nabla h(\hat{\h x}), \h x^{k_{j}+1} - \hat{\h x}\rangle. \label{T32upper1}
\end{eqnarray}
Combining (\ref{T32upper1}) with the definition of $\varphi$, we have
\begin{eqnarray}
	\varphi({\h x}^{k_{j}+1},{\h u}^{k_{j}+1},\delta) \leqslant &&\varphi(\hat{\h x},{\h u}^{k_{j}},\delta) + \theta_{k_{j}}\langle {K}^{\top}\h y^{k_{j}+1}, \h x^{k_{j}+1} - \hat{\h x}\rangle\notag \\
	&&  - \langle \nabla h(\h x^{k_{j}}) - \nabla h(\hat{\h x}), \h x^{k_{j}+1} - \hat{\h x}\rangle\notag  \\
	&& + \frac{L_{\nabla h}}{2}||\h x^{k_{j}+1} - \hat{\h x}||^{2}\notag.
\end{eqnarray}
Dividing both sides of the above equation by $f({K}\h x^{k_{j}+1})$, it leads to:
\begin{eqnarray*}
	\theta_{k_{j}+1} \leqslant &&\frac{\varphi(\hat{\h x},{\h u}^{k_{j}},\delta)}{f({K}\h x^{k_{j}+1})} + \frac{\theta_{k_{j}}}{f({K}\h x^{k_{j}+1})}\langle {K}^{\top}\h y^{k_{j}+1}, \h x^{k_{j}+1} - \hat{\h x}\rangle \\
	&& - \frac{1}{f({K}\h x^{k_{j}+1})}\langle \nabla h(\h x^{k_{j}}) - \nabla h(\hat{\h x}), \h x^{k_{j}+1} - \hat{\h x}\rangle \\
	&& + \frac{L_{\nabla h}}{2f({K}\h x^{k_{j}+1})}||\h x^{k_{j}+1} - \hat{\h x}||^{2}\notag.
\end{eqnarray*}
Note that $\varphi(\hat{\h x},\h u, \delta)$ is continuous with respect to $\h u$, $f$ is continuous on $\cal S$, and $f({K}\h x^{k_{j}+1})>m$, $\{\h y^{k_{j}+1}\}$ is bounded and $\lim_{j\rightarrow\infty}\h x^{k_{j}+1} = \hat{\h x}, \lim_{j\rightarrow\infty}\h u^{k_{j}} = \hat{\h u}$,
$\lim_{j\rightarrow\infty}\|\nabla h(\h x^{k_{j}}) - \nabla h(\hat{\h x})\|=0$ and $\lim_{j\rightarrow\infty} \theta_{k_{j}+1} = \bar{\theta}$. Consequently, as $j\rightarrow\infty$, we obtain:
\begin{eqnarray}
	\bar{\theta} \leqslant \vartheta({\hat{\h x}},{\hat{\h u}},\delta). \label{objpro2}
\end{eqnarray}
Then assertion (ii) follows from (\ref{objpro}) and (\ref{objpro2}).
%\\
\\
(iii) Invoking the optimality condition of (\ref{SplitI}), we have that
\begin{eqnarray*}\left\{\begin{array}{l}{\h y}^{k+1}\in \partial f({K}{\h x}^k),\\[0.1cm]	
{{\bf 0}\in \partial \iota _{\cal S}({{\h x}^{k+1}}) + \partial g({\h x}^{k+1})+({\h x}^{k+1}-{\h u}^k)/\delta+\nabla h({\h x}^k)-\theta_k {K}^\top {\h y}^{k+1}},\\[0.1cm]
{\h u}^{k+1}={\h u}^k-\sigma({\h u}^k-{\h x}^{k+1}).\end{array}\right.
\end{eqnarray*}
Let $\overline{\bm W}=(\overline{\h x},{\overline{\h y}},{\overline{\h u}})\in \Omega$ and let $\{{\bm W}^{k_j}\}$ be a subsequence converging  to $\overline{\bm W}$.
Then,  $\overline{\bm W}$ satisfies
\begin{eqnarray*}\left\{\begin{array}{l}{\overline{\h y}}\in \partial f({K}{\overline{\h x}}),\\
{{\bf 0}\in \partial \iota _{\cal S}({\overline{\h x}}) + \partial g({\overline{\h x}})+\nabla h({\overline{\h x}})-{\overline\theta} {K}^\top {\overline{\h y}}},\\
{\overline{\h x}}={\overline{\h u}}.\end{array}\right.
\end{eqnarray*}
It leads to
${\bf 0}\in  \partial \iota _{\cal S}({\overline{\h x}})+ \partial g({\overline{\h x}})+\nabla h({\overline{\h x}})-{\overline\theta} {K}^\top \partial f({K}{\overline{\h x}}).$
From (\ref{objpro}), we see that
$${\overline{\theta}}=\vartheta({\overline{\h x}},{\overline{\h u}},\delta)=\frac{g({\overline{\h x}})+h({\overline{\h x}})}{f({K}{\overline{\h x}})}.$$
%It further leads to
%\begin{eqnarray}\label{thestat}{\overline{\theta}}=\frac{g({\overline{\h x}})+h({\overline{\h x}})}{f(K{\overline{\h x}})}.\end{eqnarray}
 Combining these together, we see that ${\overline{\h x}}$  be a (limiting) lifted stationary point of (\ref{PForm}).\\
(iv)
 Invoking the fact of
$$\langle {\h y}^{k+1},{K}({\h x}^k-{\h x}^{k+1})\rangle =f({K}{\h x}^k)-\left(\langle {K}{\h x}^{k+1},{\h y}^{k+1} \rangle-f^*({\h y}^{k+1})\right),$$
then dividing both sides by $f({K}{\h x}^k)$ ($f({K}{\h x}^k)>m>0$),
we have that
\begin{eqnarray*}&& \frac{|f({K}{\h x}^k)-\left[ \langle {K}{\h x}^{k+1},{\h y}^{k+1}\rangle -f^*({\h y}^{k+1})\right]|}{f({K}{\h x}^k)}\nn\\
&&=\frac{|\langle {\h y}^{k+1},{K}({\h x}^k-{\h x}^{k+1})\rangle |}{f({K}{\h x}^k)}\le \frac{|\langle {\h y}^{k+1},{K}({\h x}^k-{\h x}^{k+1})\rangle |}{m}.\end{eqnarray*}
Then,  using ${\h x}^k-{\h x}^{k+1}\rightarrow 0$ and the boundedness of $\{\h y^k\}$, it leads to
\begin{eqnarray*}\label{fraclim}\lim_{k\to+\infty}\frac{|f({K}{\h x}^k)-\left[ \langle {K}{\h x}^{k+1},{\h y}^{k+1}\rangle -f^*({\h y}^{k+1})\right]|}{f({K}{\h x}^k)}=0.\end{eqnarray*}
Thus, the assertion (iv) follows directly.
%\qed
\end{proof}

Below, we present a uniform lower and upper bounds for the sequence $\{\nu^k\}$ defined by
\begin{eqnarray*}\label{nu} \nu^k:=\langle {K}{\h x}^k, {\h y}^k \rangle -f^*({\h y}^k).\end{eqnarray*}
\begin{lemma}\label{positive}Let the sequence $\{{\bm W}^k\}$ be generated by (\ref{SplitI}). Suppose that Assumption \ref{ass1}  holds.
Let $m$ and $M$ be the two positive scalars as indicated in Remark \ref{remk1}. Then,  it holds that $0<m\le \nu^k\le f({K}{\h x}^k)\le M$ for sufficiently large $k$.
\end{lemma}

\begin{proof} By invoking the assertion (iv) of Theorem \ref{subsequentialtoStat},
$\frac{\nu^{k+1}}{f({K}{\h x}^k)}\rightarrow 1 \;\mbox{as}\; k\to +\infty$
 and $f({K}{\h x}^k)>m$ for all $k$, then there exists an index ${\hat K}$ such that
 $ \nu^k \ge m$ when $k\ge{\hat K}$.
\end{proof}
Next, let $m >0$ be the scalar introduced in Remark \ref{remk1},  $\delta$ is defined in (\ref{SplitI}), and the following merit function $\varpi:\{({\h x},{\h y}) \in {\mathbb R}^n \times \text{\rm{dom}} f^* : \langle  {K} {\h x}, {\h y}\rangle -f^*({\h y})> m/2 \}  \times {\mathbb R^n}\rightarrow \overline{{\mathbb \R}}$ defined by
\begin{align*} \varpi ({\h x},{\h y},{\h u})
= & \ \frac{\varphi({\h x},{\h u})}{\langle {K}{\h x},{\h y}\rangle-f^*({\h y})},\end{align*}
where $\varphi ({\h x},{\h u})=\varphi ({\h x},{\h u},\delta)$ is used for brevity.
%where
%$$\Delta({\h x},{\h z},{\h u},{\h v})=\langle{\h z},
%A{\h x}\rangle-g^*({\h z})+h({\h x})+\frac{\delta}{2}\|{\h x}-{\h u}\|^2-
%\frac{\alpha}{2}\|{\h z}-{\h v}\|^2.$$
\begin{theorem}\label{subsequential}
Let the sequence $\{{\bm W}^k\}$ be generated by (\ref{SplitI}). Suppose that Assumption \ref{ass1} holds.
The parameters   $0<\sigma<2$ and $0<\delta<1/{L_{\nabla h}}$ are defined in (\ref{SplitI}).
  Let  $\Omega$ represent
the set of accumulation points of $\{{\bm W}^k\}$.
Then,
\begin{itemize}
%\item[(i)] any accumulation point of $\{{\bm W}^k\}$ is a (limiting) strong lifted stationary point of (\ref{PForm});
\item[(i)] there exists $c>0$  and for sufficiently large $k$ \begin{eqnarray*}\label{gammaDesc}
&&\varpi ({\h x}^{k+1},{\h y}^{k+1},{\h u}^{k+1})\le  \varpi ({\h x}^k,{\h y}^k,{\h u}^k)
-c\|{\h x}^k-{\h x}^{k+1}\|^2-c\|{\h u}^k-{\h u}^{k+1}\|^2;
\end{eqnarray*}
\item[(ii)]
$\lim_{k\to+\infty} \varpi ({\h x}^{k+1},{\h y}^{k+1},{\h u}^{k+1})$ exists
and equals to $\overline\theta$;
\item[(iii)] $\varpi \equiv{\overline\theta}$ on $\Omega$.
\end{itemize}
\end{theorem}
\begin{proof}
(i)
 Invoking Theorem \ref{PriTheo} %leads to
% \begin{eqnarray*}
% &&\varphi({\h x}^{k+1},{\h z}^{k+1},{\h u}^{k+1})+\theta_k f(K{\h x}^k)-\theta_k\left( \langle K{\h x}^{k+1},{\h y}^{k+1} \rangle-f^*({\h y}^{k+1})\right)\nn\\
%&&\le \varphi({\h x}^k,{\h z}^k,{\h u}^k)
%-c_1\|{\h x}^k-{\h x}^{k+1}\|^2-c_2\|{\h u}^k-{\h u}^{k+1}\|^2.
% \end{eqnarray*}
and the fact that $\theta_k f({K}{\h x}^k)=\varphi({\h x}^k,{\h u}^k)$, we obtain that
\begin{eqnarray}\label{inequetat}
 &&\varphi({\h x}^{k+1},{\h u}^{k+1})\le \theta_k\left( \langle {K}{\h x}^{k+1},{\h y}^{k+1} \rangle-f^*({\h y}^{k+1})\right)\nn\\
&&
-c_1\|{\h x}^k-{\h x}^{k+1}\|^2-c_2\|{\h u}^k-{\h u}^{k+1}\|^2.
 \end{eqnarray}
Invoking Lemma \ref{positive}, $\nu^{k+1}=\langle {K}{\h x}^{k+1}, {\h y}^{k+1} \rangle -f^*({\h y}^{k+1})\ge m>0$
when $k\ge {\hat K}$. Dividing both sides of (\ref{inequetat}) by $\nu^{k+1}$,
we have that ($\eta_{k+1}:=\frac{\varphi^{k+1}}{\nu^{k+1}}$ where $\varphi^{k+1}:=\varphi({\h x}^{k+1},{\h u}^{k+1})$)
\begin{eqnarray}\label{Theo4:thetak}&&\eta_{k+1}\le \theta_k
-\frac{c_1}{M}\|{\h x}^k-{\h x}^{k+1}\|^2-\frac{c_2}{M}\|{\h u}^k-{\h u}^{k+1}\|^2.
\end{eqnarray}
Since  $\theta_k\ge 0$,
$\theta_k\le \eta_k$.
The assertion then follows immediately by setting ${c}:=\min(\frac{c_1}{M}, \frac{c_2}{M})$.\\
(ii) Since  $\eta_k\ge\theta_k\ge 0$, $\{\eta_k\}$ is  lower-bounded. This together the fact that $\{\eta_k\}$ is nonincreasing implies that $\lim_{k\to+\infty} \eta_k$ exists.
 We now show that this limit equals to $\overline\theta$.
 First, invoking (\ref{Theo4:thetak}),
 it yields that
$\lim_{k\to\infty}\eta_{k+1}\le
 \lim_{k\to\infty} \theta_k.$
 On the other hand,
 $\theta_k\le \eta_k$.
 Combining these, we have   $\lim_{k\to+\infty} \eta_k={\overline\theta}$.
Thus, the assertion (ii) follows.\\
(iii)
Suppose that ${\hat {\bm W}}=({\hat{\h x}},{\hat {\h y}},{\hat{\h u}})\in \Omega$. Then, there exists a subsequence $\{{\bm W}^{k_j}\}$ such that
$\lim_{j\to\infty}\|{\bm W}^{k_j}-{\hat {\bm W}}\|=0.$
Note that
\begin{eqnarray}\label{Theo4:key1} \lim_{j\to+\infty} \frac{\varphi({\h x}^{k_j+1},{\h u}^{k_j+1})}{f({K}{\h x}^{k_j+1})}={\overline{\theta}}\end{eqnarray}
 due to
  $\lim_{k\to+\infty} \theta_k={\overline\theta}$.
  In addition,
  \begin{eqnarray}\label{Theo4:key2}\lim_{j\to+\infty}\frac{f({K}{\h x}^{k_j+1})}{f({K}{\h x}^{k_j})}=1\end{eqnarray} due to
   ${K}{\h x}^k\in \text{\rm int}({\text{\rm dom}}f)$ and $\|{\h x}^{k_j}-{\h x}^{k_j+1}\|\to 0$.

\noindent Consequently, we have that
\begin{eqnarray*} &&\lim_{j\to+\infty} \frac{\varphi({\h x}^{k_j+1},{\h u}^{k_j+1})}{\nu^{k_j+1}} \nn\\
%&&=\lim_{j\to+\infty} \frac{\varphi({\h x}^{k_j+1},{\h z}^{k_j+1},{\h u}^{k_j+1})}{f(K{\h x}^{k_j+1})}\frac{f(K{\h x}^{k_j+1})}{f(K{\h x}^{k_j})}\frac{f(K{\h x}^{k_j})}{\nu^{k_j+1}}\nn\\
&&=\lim_{j\to+\infty} \frac{\varphi({\h x}^{k_j+1},{\h u}^{k_j+1})}{f({K}{\h x}^{k_j+1})}\times\lim_{j\to+\infty}\frac{f({K}{\h x}^{k_j+1})}{f({K}{\h x}^{k_j})}\times\lim_{j\to+\infty}\frac{f({K}{\h x}^{k_j})}{\nu^{k_j+1}}
=\overline{\theta}.\end{eqnarray*}
The last equality follows from (\ref{Theo4:key1}), (\ref{Theo4:key2}) and the
  assertion (iv) of Theorem \ref{subsequentialtoStat}.
%$$\lim_{k\to+\infty}\frac{\varphi({\h x}^{k},{\h z}^{k},{\h u}^{k},{\h v}^{k})}{f(K{\h x}^{k})}$$ exists.
Consequently, we show that assertion (iii) holds. %\qed
%By combining Theorem \ref{subsequentialtoStat} (ii), (iv) and $\|{\h v}^k-{\h v}^{k+1}\|\to 0$, the assertion follows directly. \qed
\end{proof}

\section{Global convergence}\label{sec6}

In this section, we establish the convergence of the sequence generated by FPSA under mild assumptions. To do so, we present two distinct settings in which we can bound the distance between the origin and the limiting subdifferential of ${\rm\varpi}$ and ${\rm\vartheta}$, respectively, in terms of the residual between two consecutive iterates.
\vspace{-0.1cm}
\subsection{\texorpdfstring{$f^*$}{fstar} satisfies the calm condition}\label{subsec61}
\begin{theorem}\label{residual}
Let the sequence $\{{\bm W}^k\}$ be generated by (\ref{SplitI}). Suppose that Assumptions \ref{ass1} and \ref{ass2} hold.
The parameters of  $\sigma$ and $\delta$ are defined in (\ref{SplitI}) satisfying $0<\sigma<2$ and $0<\delta<1/{L_{\nabla h}}$.
Suppose that $f^*$ is
calm at ${\h y}^k$ (for a sufficiently large $k$).
 Then there exists  $\zeta>0$ and $K_0>0$, such that
for any $k\ge K_0$,
\begin{eqnarray*}&&{\text{\rm dist}}(\h 0,\partial\varpi ({\h x}^{k+1},{\h y}^{k+1},{\h u}^{k+1}))
\le \zeta(\|{\h x}^k-{\h x}^{k+1}\|+\|{\h u}^k-{\h u}^{k+1}\|).\end{eqnarray*}
\end{theorem}
\begin{proof}  There exists an index $K_0$ such that $\nu^k\ge m>m/2$, and $f^*$ is
calm at ${\h y}^k$ when $k\ge K_0$.
 Note that
\begin{eqnarray*}{\text{\rm dist}}(\h 0,\partial\varpi ({\h x}^{k+1},{\h y}^{k+1},{\h u}^{k+1}))
\le
{\text{\rm dist}}(\h 0,{\hat\partial}\varpi ({\h x}^{k+1},{\h y}^{k+1},{\h u}^{k+1})).\end{eqnarray*}
\noindent Let
\begin{equation*}\label{diffwRv}
 \begin{split}
 &(\pmb{\xi}_{\h x}^{k+1},\pmb{\xi}_{\h y}^{k+1},\pmb{\xi}_{\h u}^{k+1})\in {\hat\partial}{\varpi} ({\h x}^{k+1},{\h y}^{k+1},{\h u}^{k+1}).
 \end{split}
 \end{equation*}

 Applying Fermat's rule to (\ref{xsub}),  we have that
 $$0\in {\hat\partial} \iota _{\cal S}({\h x}^{k+1})+\partial g({\h x}^{k+1})+\nabla h({\h x}^k)-\theta_k {K}^\top {\h y}^{k+1} +({\h x}^{k+1}-{\h u}^k)/\delta.$$
 It implies that $-[\partial g({\h x}^{k+1})+\nabla h({\h x}^k)-\theta_k {K}^\top {\h y}^{k+1} +({\h x}^{k+1}-{\h u}^k)/\delta]\in {\hat\partial} \iota _{\cal S}({\h x}^{k+1})$. Set ${\bm \upsilon}^{k+1}=-[\partial g({\h x}^{k+1})+\nabla h({\h x}^k)-\theta_k {K}^\top {\h y}^{k+1} +({\h x}^{k+1}-{\h u}^k)/\delta]$.
\noindent Invoking Lemma \ref{sharpcalm} with $\nu:=m/2$, we let
 \begin{equation}\label{diff}
 \begin{split}
 &\pmb{\xi}_{\h x}^{k+1}=\frac{[{\bm \upsilon}^{k+1}+\partial g({\h x}^{k+1}) +\nabla h({\h x}^{k+1})+ ({\h x}^{k+1}-{\h u}^{k+1})/\delta]\nu^{k+1}-\varphi^{k+1}{K}^\top {\h y}^{k+1}}{(\langle {K}{\h x}^{k+1},{\h y}^{k+1}\rangle-f^*({\h y}^{k+1}))^2},\\
 &\pmb{\xi}_{\h y}^{k+1}=\frac{-{K}{\h x}^{k+1}+{{\bm \chi}^{k+1}}}{(\langle {K}{\h x}^{k+1},{\h y}^{k+1}\rangle-f^*({\h y}^{k+1}))^2}\varphi^{k+1},\;\mbox{where}\;{\bm \chi}^{k+1}\in\partial f^*({\h y}^{k+1})\\
 &\pmb{\xi}_{\h u}^{k+1}=\frac{({\h u}^{k+1}-{\h x}^{k+1})/\delta}{\langle {K}{\h x}^{k+1},{\h y}^{k+1}\rangle-f^*({\h y}^{k+1})},
 \end{split}\end{equation}
 where $\varphi^{k+1}=\varphi({\h x}^{k+1},{\h u}^{k+1})$.
Consequently, $\pmb{\xi}_{\h x}^{k+1}\in {\hat\partial}_{\h x}{\varpi} ({\h x}^{k+1},{\h y}^{k+1},{\h u}^{k+1})$. We only need to get an estimation of $\|\pmb{\xi}_{\h x}^{k+1}\|$ because
 $${\text{dist}}({\bf 0},{\hat\partial_{\h x}} \varpi({\h x}^{k+1},{\h y}^{k+1},{\h u}^{k+1}))\le \|\pmb{\xi}_{\h x}^{k+1}\|.$$
 On the other hand, we obtain that
$\pmb{\xi}_{\h x}^{k+1}=\frac{\Sigma^{k+1}}{(\nu^{k+1})^2},$
 with
\begin{eqnarray}\label{SigmaxIneq} \Sigma^{k+1}&&=[\nabla h({\h x}^{k+1})-\nabla h({\h x}^k)+({\h u}^k-{\h u}^{k+1})/\delta+\theta_k {K}^\top {\h y}^{k+1}]\nn\\
&&\times[\langle {K}{\h x}^{k+1}, {\h y}^{k+1} \rangle -f^*({\h y}^{k+1})]-{K}^\top {\h y}^{k+1}\theta_{k+1} f({K}{\h x}^{k+1})\nn\\[0.2cm]
 &&=\underbrace{[\nabla h({\h x}^{k+1})-\nabla h({\h x}^k)+({\h u}^k-{\h u}^{k+1})/\delta]\times [\langle {K}{\h x}^{k+1},{\h y}^{k+1}\rangle-f^*({\h y}^{k+1})]}_{\rm{\bf I}} \nn\\[0.2cm]
  &&+\underbrace{\theta_k {K}^\top {\h y}^{k+1}[\langle {K}{\h x}^{k+1},{\h y}^{k+1}\rangle-f^*({\h y}^{k+1})]-{K}^\top {\h y}^{k+1}{\theta}^{k+1}f({K}{\h x}^{k+1})}_{\rm{\bf II}}.\end{eqnarray}
  Before bounding ${\rm{\bm I}}$ and ${\rm{\bm II}}$,
 we first present several properties.

(a) Let  $L_h$ denote the Lipschitz constant of $h$. Set $B_{\h x}:=\sup\limits_{k\in{\mathbb N}}\|{\h x}^k\|$, $B_{\h y}:=\sup\limits_{k\in{\mathbb N}}\|{\h y}^k\|$,
 %$B_{\h z}:=\sup\limits_{k\in{\mathbb N}}\|{\h z}^k\|$,
 $B_{\h u}:=\sup\limits_{k\in{\mathbb N}}\|{\h u}^k\|$ and
 all these $B_{\ast}$ are finite ($\ast={\h x},{\h y},{\h u}$).
 $B_{\nu}:=\sup\limits_{k\in{\mathbb N}}|{\nu}^k|<+\infty$.
 Note $\varphi^{k+1}=\theta_{k+1}f({K}{\h x}^{k+1})$.
  Since $\theta_k$ and $f({K}{\h x}^k)$ are bounded,  $B_\varphi:=\sup\limits_{k\in{\mathbb N}}|\varphi^k|<+\infty$.

   (b) Further, as  $\{{\h x}^k\} \subseteq (\dom g) \cap {\cal S}$, $g$ is Lipschitz continuous on the closure of $\{{\h x}^k\}$. We denote by $L_{g}$ the corresponding Lipschitz constant. $h$ is Lipschitz continuous on ${\cal S}$; denote its Lipschitz constant $L_{h}$.

(c)
  Furthermore,
  \begin{eqnarray*}\left|\|{\h x}^{k+1}-{\h u}^{k+1}\|^2-\|{\h x}^k-{\h u}^k\|^2\right|
 \le2(B_{\h x}+B_{\h u})\left(\|{\h x}^{k+1}-{\h x}^k \|+\|{\h u}^{k+1}-{\h u}^k \|\right).
  \end{eqnarray*}
  Combining the above properties, we obtain that

  \begin{eqnarray*}|\varphi^k-\varphi^{k+1}|\nn\le \bar\varrho_1\|{\h x}^k-{\h x}^{k+1}\|
  %+\varrho_2 \|{\h z}^k-{\h z}^{k+1}\|
  +\bar\varrho_2\|{\h u}^{k+1}-{\h u}^k \|.\end{eqnarray*}
  where $\bar\varrho_1=(B_{\h x}+B_{\h u})/\delta+L_h + L_g$, $\bar\varrho_2=(B_{\h x}+B_{\h u})/\delta$.

From the above properties, we have
 \begin{eqnarray}\label{boundI}\|{\rm{\bf I}}\|\le B_{\nu}(L_{\nabla h}\|{\h x}^k-{\h x}^{k+1}\|+\|{\h u}^k-{\h u}^{k+1}\|/\delta),\end{eqnarray}
  and
  \begin{eqnarray}\label{boundII} &&{\|{\rm{\bf II}}\|}=
  \left\|{K}^\top {\h y}^{k+1}\left[ \varphi^k \frac{\nu^{k+1}}{f({K}{\h x}^k)}-\varphi^{k+1}\right]\right\|\nn\\
  &&\le\underbrace{\left
  \|{K}^\top {\h y}^{k+1}\left[  \varphi^k \frac{\nu^{k+1}}{f({K}{\h x}^k)}-\varphi^{k+1} \frac{\nu^{k+1}}{f({K}{\h x}^k)}\right]\right\| }_{\bf III}
  +\underbrace{\left\|{K}^\top {\h y}^{k+1}\left[ \varphi^{k+1} \frac{\nu^{k+1}}{f({K}{\h x}^k)}-\varphi^{k+1}\right]\right\|}_{\bf IV}.\nn\\ \end{eqnarray}
\noindent Note that
\begin{eqnarray}\label{boundIII}&&\|{\rm{\bf III}}\|\le\frac{\|{K}\|B_{\h y}B_{\nu}}{m}|\varphi^{k+1}-\varphi^k|
\le \frac{\|{K}\|B_{\h y}B_{\nu}}{m}(\varrho_1\|{\h x}^k-{\h x}^{k+1}\|
  %+\varrho_2 \|{\h z}^k-{\h z}^{k+1}\|
  +\varrho_2\|{\h u}^{k+1}-{\h u}^k \| ),\nn\\\end{eqnarray}
 and
  \begin{eqnarray}\label{boundIV}&&\| {\rm{\bf IV}}\|\le \|{K}\|B_{\h y}\left|\varphi^{k+1} \frac{\langle {K}{\h x}^{k+1}, {\h y}^{k+1}\rangle-f^*({\h y}^{k+1})}{f({K}{\h x}^k)}-\varphi^{k+1}\right|\nn\\
  &&\le \frac{B_\varphi B_{\h y}^2\|{K}\|^2}{m}\|{\h x}^k-{\h x}^{k+1}\|.
  \end{eqnarray}

By combining (\ref{SigmaxIneq}), (\ref{boundI}), (\ref{boundII}),
(\ref{boundIII}) and (\ref{boundIV}), we have that
 \begin{eqnarray*} \|\Sigma^{k+1}\|\le\bar \vartheta_1\|{\h x}^k-{\h x}^{k+1}\|+\bar\vartheta_2\|{\h u}^k-{\h u}^{k+1}\|,\end{eqnarray*}
 with
  \begin{equation*}
 \begin{split}
 &\bar\vartheta_1:=B_{\nu} L_{\nabla h}+\bar\varrho_1\|{K}\| \frac{B_{\h y}B_{\nu}}{m}+\frac{B_\varphi B_{\h y}^2\|{K}\|^2}{m},\\
 &\bar\vartheta_2:=\frac{B_{\nu}}{\delta }+\bar\varrho_2\|{K}\|\frac{B_{\h y}B_{\nu}}{m}.\\
 \end{split}\end{equation*}
 It follows that
 \begin{eqnarray*} \|\pmb{\xi}_{\h x}^{k+1}\|\le \frac{\bar\vartheta_1}{m^2}\|{\h x}^k-{\h x}^{k+1}\|+\frac{\bar\vartheta_2}{m^2}\|{\h u}^k-{\h u}^{k+1}\|.\end{eqnarray*}
 Since ${K}{\h x}^{k}\in\partial f^*({\h y}^{k+1})$, we set ${\bm \chi}^{k+1}:={K}{\h x}^{k}$ in $\pmb{\xi}_{\h y}^{k+1}$ and  we have
 \begin{equation}\nn
 \begin{split}
 &\|\pmb{\xi}_{\h y}^{k+1}\|\le B_\varphi\frac{\|{K}\|}{m^2}\|{\h x}^k-{\h x}^{k+1}\|.\\
 \end{split}
 \end{equation}
 \begin{equation}\nn
 \begin{split}
 &\|\pmb{\xi}_{\h u}^{k+1}\|\le \frac{|1-\sigma|}{\delta m\sigma}\|{\h u}^k-{\h u}^{k+1}\|.\\
 \end{split}\end{equation}
 Finally, we know that
 \begin{eqnarray*}&&{\text{\rm dist}}(\h 0,\partial\varpi ({\h x}^{k+1},{\h y}^{k+1},{\h u}^{k+1}))
\le \zeta(\|{\h x}^k-{\h x}^{k+1}\|+\|{\h u}^k-{\h u}^{k+1}\|),\end{eqnarray*}
where $\zeta>0$.
The conclusion follows directly.%\qed
 \end{proof}

\subsection{\texorpdfstring{$f$}{f} is  ${\cal C}^1$  over an open set containing \texorpdfstring{$K({\cal S})$}{KS}}\label{subsec62}

In this subsection, our analysis will be based on the merit function \({\rm\vartheta}\). First, we recall the relevant calculus rules. Suppose Assumptions \ref{ass1} and \ref{ass2} hold. Let \(f\) be differentiable at \(K {\h x} \in \inte(\dom f)\) for \({\h x} \in {\cal S}\). Define \(\alpha_1 := \varphi({\h x}, {\h u}, \delta)\) and \(\alpha_2 := f(K {\h x})\), and assume that \(\alpha_1 > 0\). Then, the Fr\'{e}chet subdifferential of \({\rm\vartheta}\) at \(({\h x}, {\h u})\) is given by:
\[
\hat\partial {\rm\vartheta}({\h x}, {\h u}) = \left\{ (\pmb{\xi}_{\h x}, \pmb{\xi}_{\h u}) \ \middle| \
\begin{array}{l}
\pmb{\xi}_{\h x} \in \displaystyle{\frac{\alpha_2 (\partial g({\h x}) + \nabla h({\h x}) + \partial \iota_{\cal S}({\h x}) + ({\h x} - {\h u})/\delta) - \alpha_1 K^\top \nabla f(K {\h x})}{(f(K {\h x}))^2}}\\
\pmb{\xi}_{\h u} = \displaystyle{\frac{{\h u} - {\h x}}{\delta f(K {\h x})}}
\end{array}
\right\}.
\]

 \begin{theorem}\label{residualDiffF}
Suppose that $f$ is ${\cal C}^1$ over an open set containing $K({\cal S})$. Then there exists  $\varsigma>0$ such that for all $k$ sufficiently large,
\begin{eqnarray*}&&{\text{\rm dist}}(\h 0,\partial{\rm\vartheta} ({\h x}^{k+1},{\h u}^{k+1}))
\le \varsigma(\|{\h x}^k-{\h x}^{k+1}\|+\|{\h u}^k-{\h u}^{k+1}\|).\end{eqnarray*}
\end{theorem}
\begin{proof}The proof is similar to Theorem \ref{residual}, thus omitted here. \end{proof}
% \end{lemma}
Now, we are in the stage to present the global convergence results of FPSA.

\begin{theorem} \label{theo8}
Suppose Assumption \ref{ass1} and Assumption \ref{ass2}  hold, and one of the following conditions is fulfilled:
  \begin{itemize}
  \item[{\rm (i)}] $f^*$ is
calm at ${\h y}^k$ (for a sufficiently large $k$)  and $\varpi$ possesses the KL property at every point of ${\rm dom} \, \partial (\varpi)$.
 \item[{\rm (ii)}] $f$ is differentiable with Lipschitz continuous gradient over an open set containing $K({\cal S})$ and $\vartheta$ satisfies the KL property at every point of ${\rm dom} \, \partial (\vartheta)$.
     \end{itemize}
 Then,
$$
\sum_{k} \Big(\|{\h x}^k-{\h x}^{k+1}\| + \|{\h u}^k-{\h u}^{k+1}\|  \Big) <+\infty,
$$
And $\{\h x^k\}$ converges to a  (limiting) lifted stationary point of (\ref{PForm}).
\end{theorem}
\begin{proof}
	The proof is similar to \cite[Theorem 6.7]{boct2023full}, thus omitted here.
\end{proof}
\begin{remark}
In the two settings above, we also require that either $f^*$ is
calm at ${\h y}^k$ (for a sufficiently large $k$)  or that \( f \) is differentiable with a Lipschitz continuous gradient over an open set containing \( K({\cal S}) \). We note that these conditions are satisfied in many practical applications. All examples (a)-(d) in the introduction satisfy the conditions stated in Theorem \ref{theo8}, i.e.,
(a) satisfies the assumption (i) while the others satisfy the assumption (ii).

\end{remark}

\section{Numerical Results}\label{num}
In this section, we further develop a nonmonotone line search version of the FPSA and compare its performance with several state-of-the-art algorithms. All experiments were conducted using MATLAB R2024a on a desktop running Windows 11, equipped with an Intel Core i9-13980HX CPU (2.20 GHz) and 16 GB of memory.

 To avoid the dependency
on the constant \( L_{\nabla h} \), we propose a nonmonotone line search FPSA, referred to as FPSA-{nl}. The specific algorithm is detailed in Algorithm \ref{FPSA-nlr}.
The stopping criterion throughout the numerical experiments is
\begin{equation}\label{STOPC}
	\begin{aligned}
		\frac{||\h x^{k+1} - \h x^{k}||}{\max\{||\h x^{k}||,\text{eps}\}} & < {\tt Tol}\;
		\text{or}\;k> {\tt MaxIt},
	\end{aligned}		
\end{equation}
where eps denotes the machine precision.

\begin{algorithm}[ht]
	\caption{FPSA-{nl}}	
	\begin{algorithmic}[1]
		\Require{ $\h x^{0}$, ${\h u}^0$,  $0<\sigma<2$, $\rho_{1}>0$, $0<q<1$, $\varepsilon>0$, $N$, $T$ and {\tt MaxIt}.}
		\For\; $k$ = 0:{\tt MaxIt}
%		\State{Let $\delta_k^0=\rho\frac{||\h x^{k} - \h x^{k-1}||_{2}}{||\nabla h(\h x^{k}) -\nabla h(\h x^{k-1})||_{2}}$.}
		\State{Choose the initial step size $\delta_{k,0} > 0$.}
		\State{ ${\h y}^{k+1}\in\partial f({K}{\h x}^k).$}
		\For  {$j = 1:N$}
			\State{$\delta_{k,j} = \delta_{k,0}*q^{j-1}$ }
			\State{$\h x^{k+1}\in\arg\min_{{\h x}\in{\cal S}} \left( g({\h x})+\frac{1}{2\delta_{k,j} }\left\|{\h x}-{\h u}^k+\delta_{k,j} \nabla h({\h x}^k)-\theta_k\delta_{k,j}  {K}^\top {\h y}^{k+1}\right\|^2\right)$}
			\State{$\theta_{k+1}=\vartheta({\h x}^{k+1},{\h u}^{k};\delta_{k,j})$}
			\If{$\theta_{k+1} < \max\limits_{[k-T]_{+}<s\leqslant k}\theta_{s} - \rho_{1}||\h x^{k+1}-\h x^{k}||^{2}$}
			\State{break}
			\EndIf
		\EndFor
		\State{${\h u}^{k+1}=(1-\sigma){\h u}^k+\sigma{\h x}^{k+1}.$}
		\If{Stopping Criterion}
		\State{break}
		\EndIf	
		\EndFor
		\State{{\bf Output} ${\h x}$.}
	\end{algorithmic}\label{FPSA-nlr}
\end{algorithm}

\subsection{$L_{1}/S_{\kappa}$ sparse signal recovery}

We solve the $L_{1}/S_{\kappa}$ regularized sparse recovery problem (\ref{L1oLK}),  and compare the FPSA-nl with two other methods: the Proximity Gradient Subgradient Algorithm with Backtracked Extrapolation (PGSA\_{BE}) from \cite{LSZ}, and the Extrapolated Proximal Subgradient (e-PSG) algorithm from \cite{BDL}. We use (\ref{STOPC}) as stopping criterion with ${\tt MaxIt} = 5000$ and ${\tt Tol} = 10^{-6}$.

We set $\lambda = 10^{-3}$, ${\h c} = -{\h e}, \, {\h d} = {\h e}$ in the model (\ref{L1oLK}). We generate the Oversampled DCT (O-DCT) matrix $A = [{\bf a}_1,{\bf a}_2,\ldots,{\bf a}_n] \in {\mathbb R}^{m \times n}$, where each column ${\bf a}_j = \frac{1}{\sqrt{m}} \cos \left( \frac{2\pi {\bf w}j}{F} \right)$ for $j=1,\ldots,n$. Here, ${\bf w} \in {\mathbb R}^m$ is a uniformly distributed random vector over $[0,1]$, and $F \in {\mathbb R}_+$ controls the coherence. Next, we generate a $\kappa$-sparse ground truth signal $\h x^* \in {\mathbb R}^n$:
(1) The support set $\Lambda$ is generated randomly with a minimum distance $2F$;
(2) We generate a nonzero vector $\h y \in \mathbb{R}^{\kappa}$ subject to a standard Gaussian distribution, and $\h x^*|_{\Lambda} = \text{sign}{(\h y)}/2$;
(3) $\h b = A \h x^{*}$.

The parameters in FPSA-nl are set to $\sigma = 1.35$, $\rho_{1} = 10^{-3}$, and
{
\begin{eqnarray}\label{del0}
\delta_{k,0} =\left\{ \begin{array}{cl} \frac{\|\h x_0\|}{\max\{\| \nabla h(\h x_0)\|,\text{eps}\}}& k=0,\\[0.1cm]
\varsigma \cdot \frac{\|\h x^{k} - \h x^{k-1}\|}{\max\{\|\nabla h(\h x^{k}) - \nabla h(\h x^{k-1})\|,\text{eps}\}} &k\ge 1,\end{array}\right.\end{eqnarray}
}
with $\varsigma=0.8$.
 %$\delta_{0,0} = \|\h x_0\| / \max\{ \|h(\h x_0)\|,\text{eps}\}$, and
%\[
%\delta_{k,0} = 0.8 \cdot ||\h x^{k} - \h x^{k-1}|| / \max\{ ||\nabla h(\h x^{k}) - \nabla h(\h x^{k-1})||,\text{eps}\}
%\]
Set $q = 0.9$, $T = 20$.

For PGSA\_{BE}, we set $f(\h x) = \lambda ||\h x||_1 + \iota_{\mathcal{S}}(\h x)$, $h(\h x) = \frac{1}{2} ||A \h x - \h b||^2$, $g(\h x) = ||\h x||_{(\kappa)}$, and we set $\epsilon = 10^{-4}$, $\alpha = 1/||A||_2^2$, $l = 0$. Then, we use a recursive sequence to update
\[
\theta_{k+1} = \frac{1 + \sqrt{1 + 4\theta_k^2}}{2}, \quad \text{where} \quad \theta_{-1} = \theta_0 = 1.
\]
We set $\beta_k = (\theta_{k-1} - 1)/\theta_k$ and reset $\theta_{k-1} = \theta_k = 1$ for every 100 iterations ($\beta_{100} \approx 0.97$). Then $\{\beta_k\} \subseteq [0, \bar{\beta}]$ for some $0 < \bar{\beta} < 1$.

For e-PSG, we set $f^{n}(\h x) = \lambda ||\h x||_1$, $f^{s}(\h x) = \frac{1}{2} ||A \h x - \h b||^2$, $g(\h x) = ||\h x||_{(\kappa)}$, and we set $\beta = 0$, $\delta = 10^{-3}$, $\ell = ||A||_2^2$, $\tau_{k} \equiv 1/\delta$. Note that (\ref{L1oLK}) doesn't satisfy the BC condition; therefore, $\bar{\mu} = \bar{\kappa} = 0$.

All these algorithms are initialized from $\h x_0 = \h x^* + 0.2 * \h z$, where $\h z$ is subjected to the uniform distribution on $[-1, 1]^{n}$.

We measure performance in terms of CPU time, the lifted stationary residual (StatRes), defined as
\begin{equation*}
    \text{StatRes}(\h x) = \text{dist}({\h 0}, (\nabla h({\h x}) + \partial g({\h x}) + \partial \iota_{\cal S}({\h x}))f({K}{\h x}) - (h({\h x}) + g({\h x})) {K}^{\top}\partial f(K{\h x})),
\end{equation*}
and the relative error (Err):
\[
    \text{Err}(\h x) = \frac{||\h x - \h x^{*}||}{\max\{||\h x^{*}||,\text{eps}\}},
\]
where $\h x$ denotes the last iterate.

\begin{figure}[htb]
	\vspace{0cm}\centering{\vspace{0cm}
		\begin{tabular}{ccc}
			\includegraphics[scale = .35]{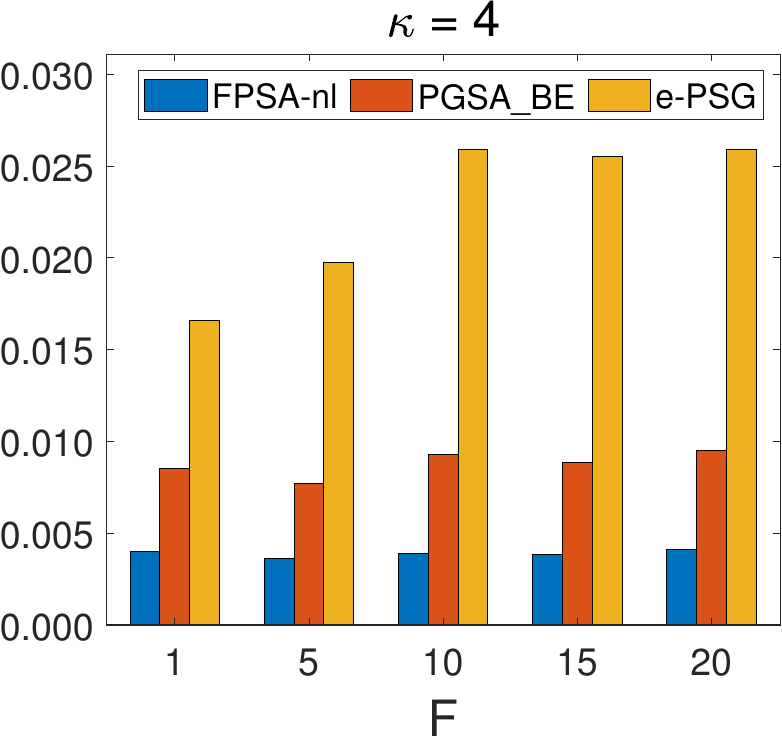}
			& \includegraphics[scale = .35]{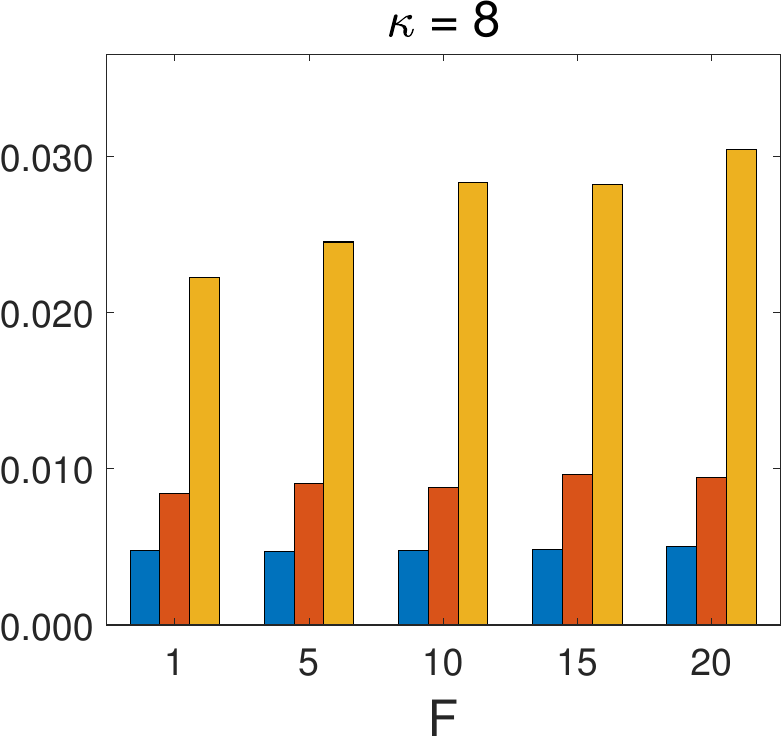}
			& \includegraphics[scale = .35]{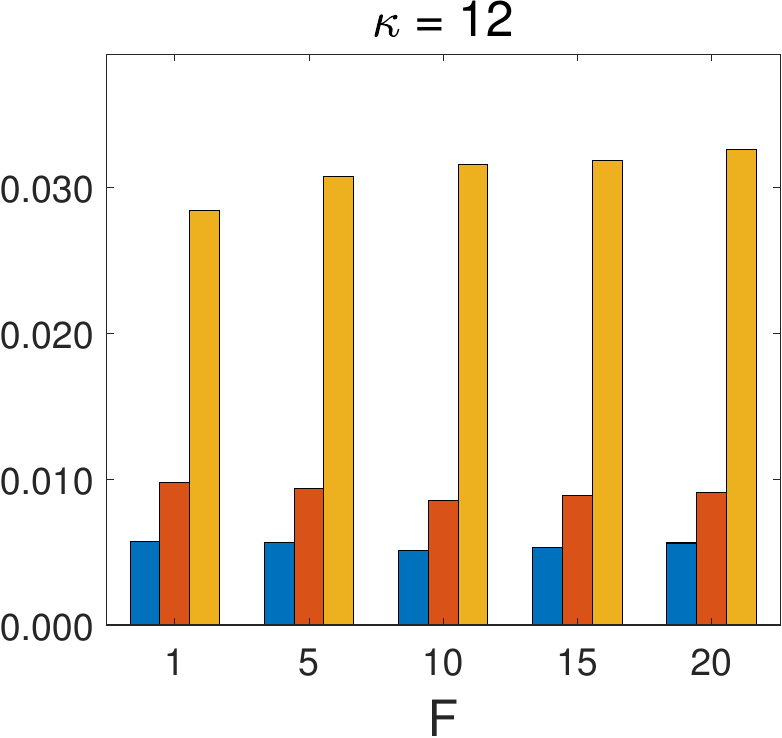}
		\end{tabular}
	} \caption{The averaged results of CPU time from three algorithms with $\kappa\in\{4, 8, 12\}$ and $F\in\{1, 5, 10, 15, 20\}$.}
	\label{L1oLK_CPU}
\end{figure}

\begin{figure}[htb]
	\vspace{0cm}\centering{\vspace{0cm}
		\begin{tabular}{ccc}
			\includegraphics[scale = .35]{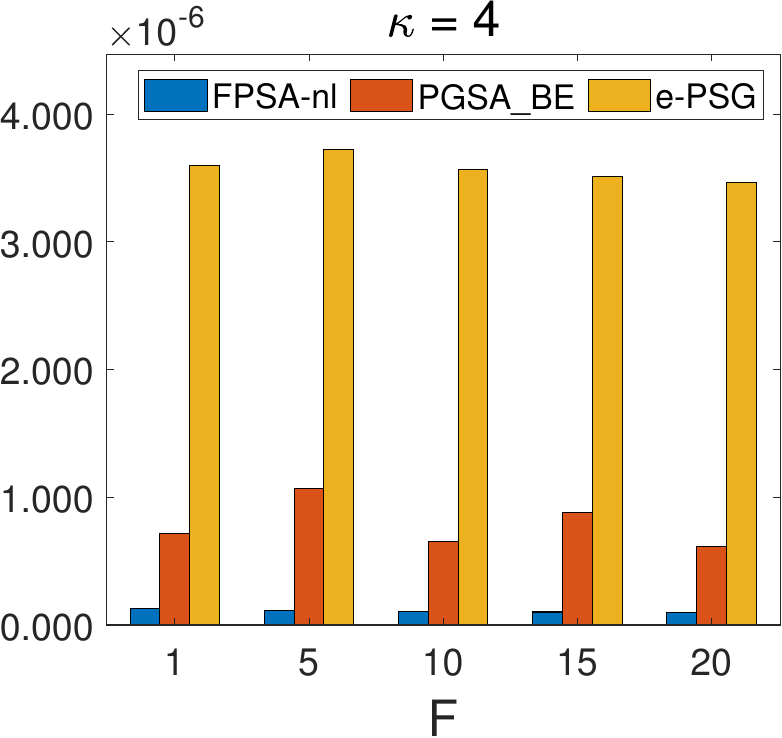}
			& \includegraphics[scale = .35]{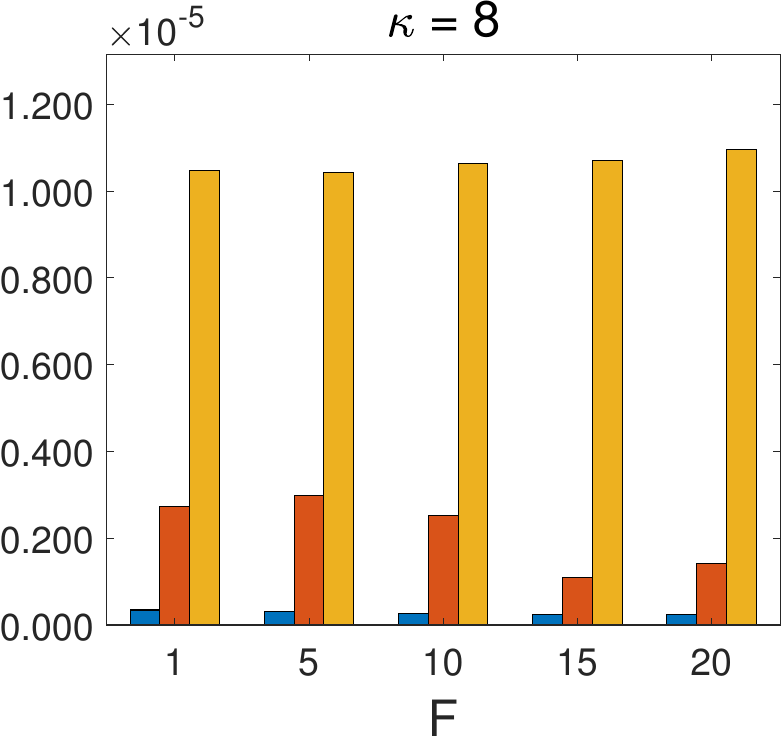}
			& \includegraphics[scale = .35]{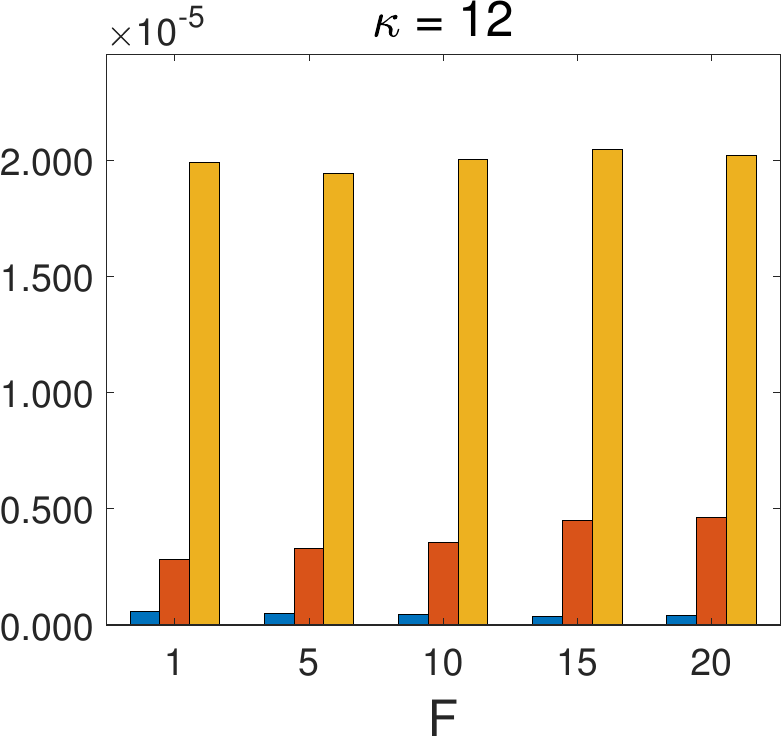}
		\end{tabular}
	} \caption{The averaged results of StatRes from  three algorithms with $\kappa\in\{4, 8, 12\}$ and $F\in\{1, 5, 10, 15, 20\}$.}
	\label{L1oLK_kkt}
\end{figure}

\begin{figure}[htb]
	\vspace{0cm}\centering{\vspace{0cm}
		\begin{tabular}{ccc}
			\includegraphics[scale = .35]{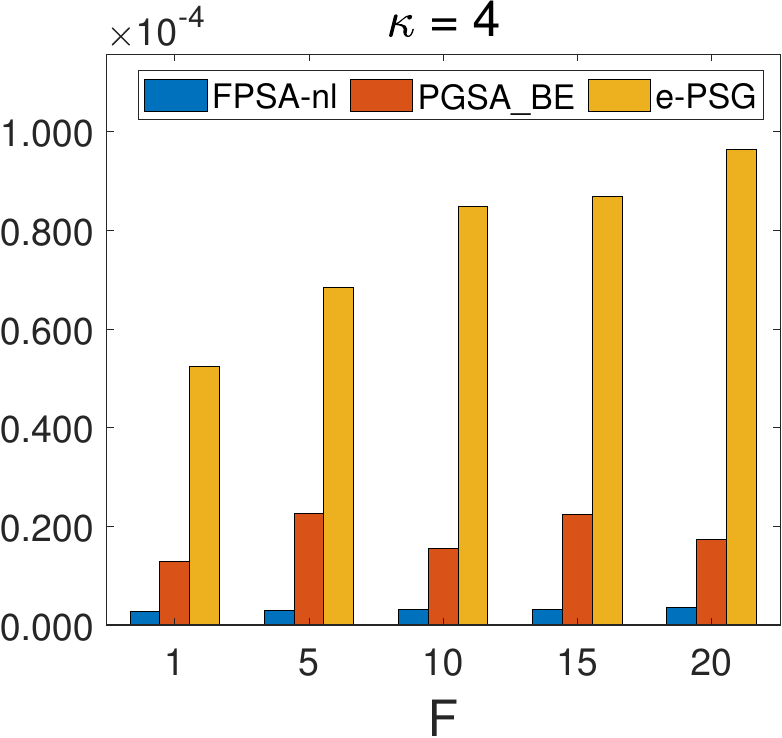}
			& \includegraphics[scale = .35]{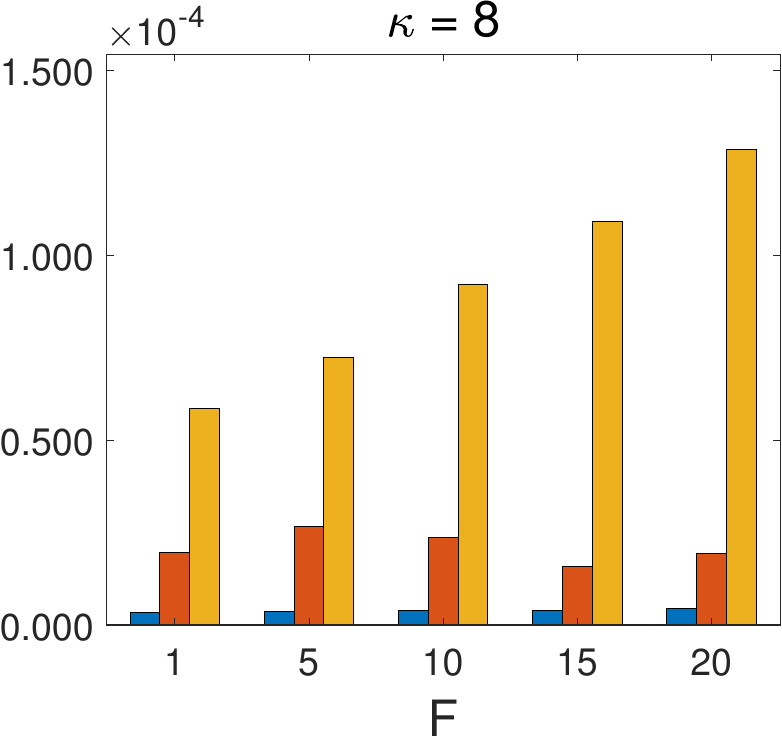}
			& \includegraphics[scale = .35]{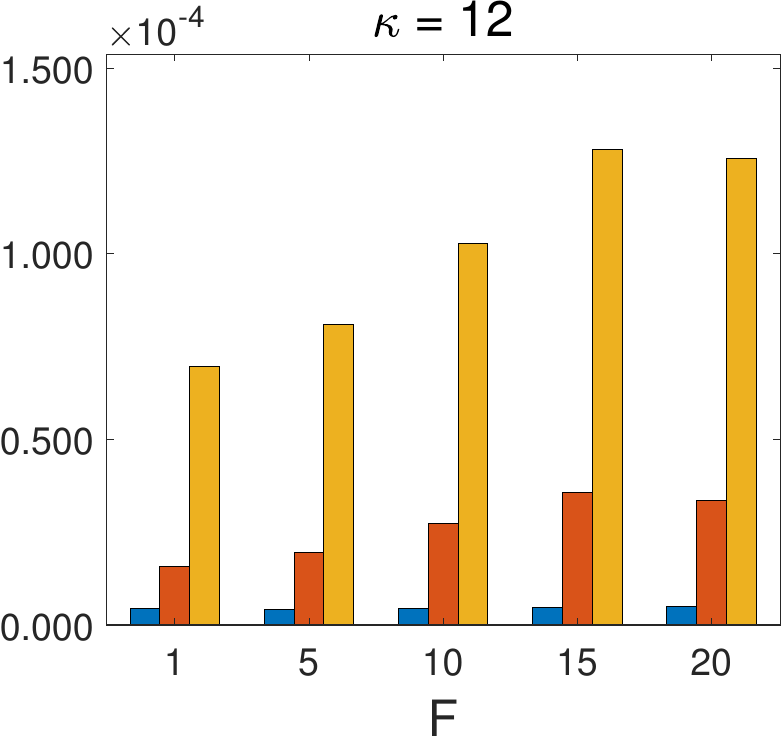}
		\end{tabular}
	} \caption{The averaged results of Err from three algorithms with $\kappa\in\{4, 8, 12\}$ and $F\in\{1, 5, 10, 15, 20\}$.}
	\label{L1oLK_err}
\end{figure}

  We test on the scenarios of $(m,\ n) = (64, \ 1024)$ with  varying sparsity $\kappa \in \{4, 8, 12\}$ and $F \in \{1, 5, 10, 15, 20\}$.
  For each scenario, we take the average of  20 repeated experiments.
In Figures \ref{L1oLK_CPU}, \ref{L1oLK_kkt}, and \ref{L1oLK_err}, we present bar charts showing the average results in terms of
 CPU time, StatRes, and Err from FPSA-nl, PGSA\_BE, and e-PSG under different sparsity levels and different values of $F$.
 Notably, FPSA-nl requires less CPU time while achieving much more accurate solutions, as evidenced by its significantly lower StatRes and Err values compared to PGSA\_BE and e-PSG.

\subsection{Limited-Angle CT Reconstruction}

We consider the limited-angle scan model for CT reconstruction  (\ref{CT}) with \({\h c} = {\bf 0}\) and \({\h d} = {\bf e}\) in \({\cal S}\). We compare the numerical performance of three algorithms: the FPSA-nl, the e-PSA \cite{BDL}, and the PGSA\_{BE} \cite{LSZ}. We use the same stopping criterion (\ref{STOPC}) with \({\tt MaxIt} = 5000\) and \({\tt Tol} = 10^{-6}\).

We use two phantoms: the Shepp-Logan (SL) and the FORBILD (FB) \cite{70}. We followed the routines outlined in \cite{WTNL} to generate the limited-angle CT data. The SL and FB phantoms were discretized at a resolution of \(128 \times 128\). The parallel beam CT scanning was simulated using the AIR and IR toolbox \cite{gazzola2019ir,hansen2012air}. The linear operator \(A \in \mathbb{R}^{5611 \times 16384}\) was formulated as the discrete Radon transform, corresponding to the phantoms' resolution. The maximum angle \(\theta_{\text{Max}}\) was set to either \(90^{\circ}\) or \(150^{\circ}\). Sampling was performed at intervals of \(\theta_{\text{Max}}/30\) over the range from \(0^{\circ}\) to \(\theta_{\text{Max}}\).

The additive noise was subjected to a standard Gaussian distribution, and the noise levels  were 0\% (noiseless), 0.1\%, and 0.5\%. The parameter $\lambda$ in (\ref{CT}) was set by hand-tuning and is provided in Table \ref{SL_FB_para}.

	\begin{table}[ht]
	\caption{The choice of $\lambda$ for SL and FB Phantoms.}
	\label{SL_FB_para}
\begin{center}
	{\vskip -1mm \centering
	\resizebox{0.75\linewidth}{0.55in}{
		\begin{tabular}{c|c|c|c|c|c|c|c}
			\hline
			\multirow{7}{*}{SL}      & Projection  					& Noise level & $\lambda$ &
			\multirow{7}{*}{FB} 		   & Projection  					& Noise level & $\lambda$ \\ \cline{2-4} \cline{6-8}
			& \multirow{3}{*}{90$^{\circ}$}  & noiseless   & 0.25      &
			& \multirow{3}{*}{90$^{\circ}$}  & noiseless   & 0.25      \\ \cline{3-4} \cline{7-8}
			&                                & 0.1\%       & 0.25      &                          &                                & 0.1\%       & 0.25      \\ \cline{3-4} \cline{7-8}
			&                                & 0.5\%       & 1.00      &                          &                                & 0.5\%       & 1.00      \\ \cline{2-4} \cline{6-8}
			& \multirow{3}{*}{150$^{\circ}$} & noiseless   & 0.25      &
			& \multirow{3}{*}{150$^{\circ}$} & noiseless   & 0.25      \\ \cline{3-4} \cline{7-8}
			&                                & 0.1\%       & 0.25      &                          &                                & 0.1\%       & 0.80      \\ \cline{3-4} \cline{7-8}
			&                                & 0.5\%       & 1.00      &                          &                                & 0.5\%       & 1.00      \\ \hline
		\end{tabular}
	}}\end{center}
\end{table}

In implementing FPSA-nl,   the ${\h x}$-subproblem is
\begin{equation*}
	\h x^{k+1} = {\arg\min}_{ \h x \in {\cal S}}\ \left[\lambda||\nabla\h x||_{1} + \frac{1}{2\delta_{k,j}}||\h x - \h z^{k}||^2 \right] \label{ADMM_sub}
\end{equation*}
where $\h z^{k} = \h u^{k} - \delta_{k,j}\nabla h(\h x^{k}) + \theta_{k}\delta_{k,j}\nabla^{\top}\h y^{k+1}$. We reformulated it as:

\begin{equation*}
	\begin{aligned}
		\min_{\h w, \h x,\h h}\ & \iota_{{\cal S}}(\h h)+\lambda||\h w||_{1} + \frac{1}{2\delta_{k,j}}||\h x - \h z^{k}||^2  \\
		s.t.\ &  \nabla \h x = \h w ,\ \h h = \h x.
	\end{aligned}
\end{equation*}

\noindent By using the alternating direction method of multiplier (ADMM), we have the following iterative scheme:

\begin{subequations}
	\begin{numcases}{\hbox{\quad}}
		{\h w}^{t+1} = \text{shrink}(\nabla \h x^{t} + \h v^{t},\lambda/\alpha),\notag\\[0cm]
		\h x^{t+1} =  \left( \alpha\nabla^{\top}\nabla + (\frac{1}{\delta_{k,j}} + \beta)I  \right)^{-1}\left( \frac{\h z^{k}}{\delta_{k,j}} - \alpha\nabla^{\top}\h v^{t} + \alpha\nabla^{\top}\h w^{t+1} - \beta \bm{\mu}^{t} + \beta \h h^{t}    \right),\notag\\
		{\h h}^{t+1}=\min\{\max\{\h x^{t+1} + \bm{\mu}^{t},\h c\},\h d\},\notag\\[0.1cm]
		\h v_{t+1}=\h v^{t} + \nabla \h x^{t+1} - {\h w}^{t+1}, \notag\\
		\bm{\mu}^{t+1} = \bm{\mu}^{t} + \h x^{t+1} - \h h^{t+1},\notag
	\end{numcases}
\end{subequations}
where $j$ denotes the inner iteration number.

Parameters setting in FPSA-nl: In the outer iteration, we set \(\sigma = 1.0\),
and set $\delta_{k,0}$  as (\ref{del0})
with $\varsigma=0.8$.
 We took \(q = 0.95\), \(T = 5\), \(\rho_{1} = 10^{-3}\), and \(N = 250\).
In the inner iteration, we set \(\alpha = 5\) and \(\beta = 0.0005\), with a maximum  iteration number 3 for the \(\h x\)-subproblem.

For PGSA\_{BE}, we took $f(\h x) = \lambda||\nabla \h x||_1 + \iota_{\mathcal{S}}(\h x)$, $h(\h x) = \frac{1}{2}||A\h x - \h b||^2$, $g(\h x) = ||\nabla \h x||_{2}$ and we set $\epsilon = 10^{-4}$, $\ell = 0$, $\beta_k \equiv 0.99$ for both SL and FB. We took $\alpha = 4\times 10^{-4}$ for SL, and  $\alpha = 3.8\times 10^{-4}$ for FB. At each iteration, PGSA\_BE also required solving a denoise subproblem with
\(\h z^k = \h u^{k+1} - \alpha\nabla h(\h u^{k+1}) + \alpha c_k \h y^{k+1}\)
or
\(\h z^k = \h x^{k} - \alpha\nabla h(\h x^{k}) + \alpha c_k \h y^{k+1}\).
We still use ADMM to solve it with \(\rho_1\) and \(\rho_2\) as penalty parameters to  \(\h w = \nabla \h x\) and \(\h h = \h x\), respectively. We set \(\rho_1 = 5\times 10^{-5}\), \(\rho_2 = 10^{-2}\), and the termination criterion of the inner loop was that the relative error between two consecutive iterations was less than \(10^{-6}\) or the inner iteration number reached 300.

For e-PSG, we took $f^{n}(\h x) = \lambda||\nabla \h x||_1$, $f^{s}(\h x) = \frac{1}{2}||A\h x - \h b||^2$, $g(\h x) = ||\nabla \h x||_{2}$, and we set $\beta = 0$, $\ell = \|A^\top A\|$. We set \(\delta = 1500\) for SL, and \(\delta = 2400\) for FB. Then we took \(\tau_k \equiv 1/\delta\) for all \(k \ge 0\). Note that (\ref{CT}) did not satisfy the BC condition. Thus, we set \(\bar{\mu} = \bar{\kappa} = 0\). At each step, e-PSG also required solving a denoise subproblem with
\(\h z^k = \h x^k + (\theta_k \h g^k - \nabla f^{s}(\h x^k))/(1/\tau_k + l)\).
We adopted ADMM to solve it. Let \(\rho_1\) and \(\rho_2\) be the penalty parameters for \(\h w = \nabla \h x\) and \(\h h = \h x\).
We set \(\rho_1 = 0.05\), \(\rho_2 = 0.01\), and the termination criterion of the inner loop was that the relative error between two consecutive iterations was less than \(10^{-6}\) or the number of iterations reached 300.
For all the solvers, the initial point was chosen as the zero matrix.

We measured performance in terms of the root mean squared error (RMSE), the overall structural similarity index (SSIM), and the CPU time (CPU). The RMSE was defined as

\[
\text{RMSE}(\h x, \h x^{*}) := \frac{\|\h x - \h x^{*}\|}{n \times n},
\]
where \(\h x^{*} \in \mathbb{R}^{n \times n}\) is the ground truth. The SSIM was defined as:
$$
\text{SSIM}({\h u},{{\h v}}) = \frac{1}{M} \sum_{i=1}^{M} {\rm ssim}(u^{(i)}, v^{(i)}),
$$
where \(u^{(i)}\) and \(v^{(i)}\) represented the \(i\)-th \(3 \times 3\) windows for \({\h{u}}\) and \({\h{v}}\), respectively, and \(M\) was the total number of such windows. The local similarity index was computed using the formula:

$$
{\rm ssim}(u^{(i)}, v^{(i)}) = \frac{(2\mu_u^{(i)}\mu_v^{(i)} + c_1)(2\sigma_{uv}^{(i)} + c_2)}{((\mu^{(i)}_u)^2 + (\mu^{(i)}_v)^2 + c_1)((\sigma^{(i)}_u)^2 + (\sigma^{(i)}_v)^2 + c_2)},
$$
where \(\mu^{(i)}_{u}\) and \(\mu^{(i)}_{v}\) represented the averages, \((\sigma^{(i)}_{u})^2\) and \((\sigma^{(i)}_{v})^2\) denoted the variances of \(u^{(i)}\) and \(v^{(i)}\), and \(\sigma^{(i)}_{uv} = \frac{1}{63} \sum (u^{(i)} - \mu_u^{(i)})(v^{(i)} - \mu_v^{(i)})\). The constants \(c_1\) and \(c_2\) were set to \(0.05\) to stabilize the division with a weak denominator. We set \(c_{1} = c_{2} = 0.05\).

	\begin{figure}[htb]
		\vspace{0cm}\centering{\vspace{0cm}
			\begin{tabular}{cccc}
				\includegraphics[scale = .5]{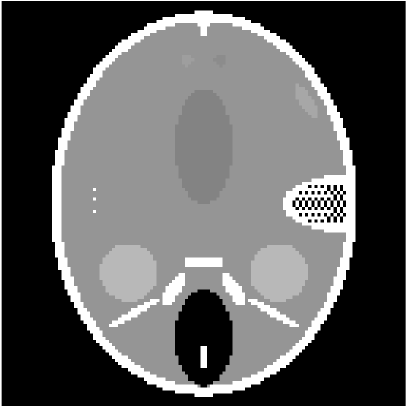}
				& \includegraphics[scale = .5]{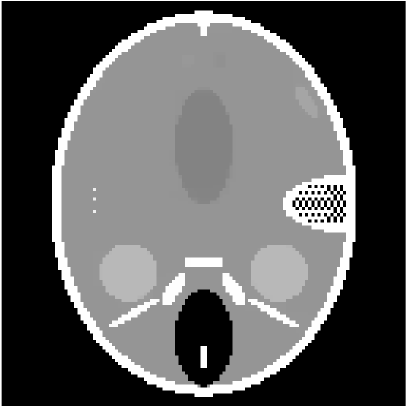}
				& \includegraphics[scale = .5]{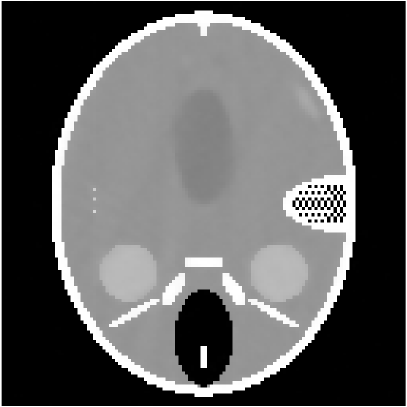}
				& \includegraphics[scale = .5]{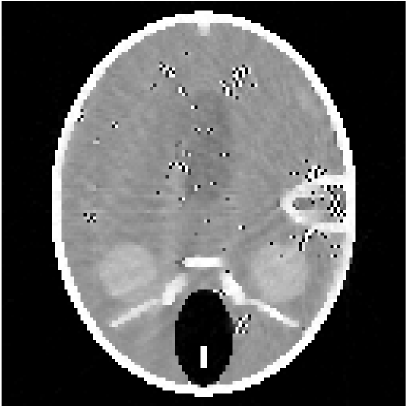}
			\end{tabular}
		}
		\caption{The original images of FB, as well as the reconstructed results from FPSA-nl, PGSA\_BE, e-PSG (from left to right)  under the projection angle of 150$^{\circ}$ and noise level at 0.1\%. }
		\label{SL_FB}
	\end{figure}

\begin{table}[ht]
	\caption{ CT reconstruction results of the Shepp-Logan  and FORBILD.}
			\label{SL_tabular}
			\renewcommand{\arraystretch}{1.15}
	\begin{tabular}{cccccccc}
		\toprule
\multicolumn{8}{c}{Shepp-Logan}\\
\toprule
		\multicolumn{2}{c}{noise level}      & \multicolumn{2}{c}{noiseless} & \multicolumn{2}{c}{0.1\%}    & \multicolumn{2}{c}{0.5\%}    \\
		\cmidrule[\heavyrulewidth](lr){1-2} \cmidrule[\heavyrulewidth](lr){3-4} \cmidrule[\heavyrulewidth](lr){5-6} \cmidrule[\heavyrulewidth](lr){7-8}
		\multicolumn{2}{c}{Projection Angle} & $90^{\circ}$  & $150^{\circ}$ & $90^{\circ}$ & $150^{\circ}$ & $90^{\circ}$ & $150^{\circ}$ \\ \toprule[0.8pt]
		\multirow{3}{*}{SSIM}
%		& SART       & 8.77e-01      & 9.16e-01      & 8.76e-01     & 9.15e-01      & 8.51e-01     & 8.89e-01      \\
		& PGSA\_BE   & 9.99e-01      & 1.00e+00      & 9.99e-01     & 1.00e+00      & 9.98e-01     & 9.99e-01      \\
		& e-PSG      & 9.08e-01      & 9.99e-01      & 9.06e-01     & 9.98e-01      & 8.57e-01     & 9.97e-01      \\
		& FPSA-nl    & 1.00e+00      & 1.00e+00      & 1.00e+00     & 1.00e+00      & 9.99e-01     & 1.00e+00      \\ \hline
		\multirow{3}{*}{RMSE}
%		& SART       & 1.00e-03      & 7.97e-04      & 1.00e-03     & 8.02e-04      & 1.10e-03     & 9.19e-04      \\
		& PGSA\_BE   & 5.77e-05      & 2.00e-05      & 5.81e-05     & 2.12e-05      & 9.59e-05     & 5.87e-05      \\
		& e-PSG      & 8.29e-04      & 6.27e-05      & 8.39e-04     & 6.37e-05      & 1.04e-03     & 1.06e-04      \\
		& FPSA-nl    & 9.94e-06      & 5.29e-06      & 2.90e-05     & 1.73e-05      & 6.94e-05     & 3.18e-05      \\ \hline
		%\multirow{3}{*}{PSNR}
%%		& SART       & 3.57e+01      & 3.36e+01      & 3.57e+01     & 3.36e+01      & 3.57e+01     & 3.36e+01      \\
%		& PGSA\_BE   & 3.10e+01      & 3.10e+01      & 3.10e+01     & 3.10e+01      & 3.11e+01     & 3.10e+01      \\
%		& e-PSG      & 3.31e+01      & 3.10e+01      & 3.31e+01     & 3.10e+01      & 3.33e+01     & 3.11e+01      \\
%		& FPSA-nl    & 3.09e+01      & 3.09e+01      & 3.09e+01     & 3.09e+01      & 3.10e+01     & 3.10e+01      \\ \hline
		\multirow{3}{*}{CPU}
%		& SART       & 4.22e+00      & 3.99e+00      & 4.16e+00     & 4.22e+00      & 4.14e+00     & 4.19e+00      \\
		& PGSA\_BE   & 6.37e+02      & 5.32e+02      & 6.54e+02     & 4.83e+02      & 8.39e+02     & 7.14e+02      \\
		& e-PSG      & 1.84e+03      & 1.87e+03      & 1.87e+03     & 1.88e+03      & 2.38e+03     & 2.39e+03      \\
		& FPSA-nl    & 1.34e+01      & 1.24e+01      & 1.99e+01     & 6.91e+00      & 1.23e+01     & 7.31e+00      \\
	\toprule
\multicolumn{8}{c}{FORBILD}\\
		\toprule
		\multirow{3}{*}{SSIM}
%		& SART     & 8.44e-01      & 9.09e-01      & 8.39e-01     & 7.58e-01      & 7.58e-01     & 8.03e-01      \\
		& PGSA\_BE & 9.99e-01      & 1.00e-01      & 9.99e-01     & 9.87e-01      & 9.85e-01     & 9.97e-01      \\
		& e-PSG    & 7.83e-01      & 8.26e-01      & 7.78e-01     & 9.39e-01      & 7.90e-01     & 8.53e-01      \\
		& FPSA-nl  & 1.00e+00      & 1.00e+00      & 1.00e+00     & 9.90e-01      & 9.90e-01     & 9.98e-01      \\ \hline
		\multirow{3}{*}{RMSE}
%		& SART     & 1.20e-03      & 8.50e-04      & 1.21e-03     & 8.76e-04      & 1.52e-03     & 1.36e-03      \\
		& PGSA\_BE & 5.62e-05      & 1.72e-05      & 6.03e-05     & 5.70e-05      & 2.89e-04     & 1.09e-04      \\
		& e-PSG    & 1.36e-03      & 1.07e-03      & 1.38e-03     & 1.75e-05      & 1.37e-03     & 1.01e-03      \\
		& FPSA-nl  & 2.54e-05      & 4.98e-06      & 5.34e-05     & 1.75e-05      & 2.48e-04     & 9.84e-05      \\ \hline
		%\multirow{3}{*}{PSNR}
%%		& SART     & 2.34e+01      & 2.22e+01      & 2.34e+01     & 2.22e+01      & 2.34e+01     & 2.22e+01      \\
%		& PGSA\_BE & 2.12e+01      & 2.12e+01      & 2.12e+01     & 2.22e+01      & 2.14e+01     & 2.12e+01      \\
%		& e-PSG    & 2.23e+01      & 2.12e+01      & 2.23e+01     & 2.15e+01      & 2.24e+01     & 2.16e+01      \\
%		& FPSA-nl  & 2.12e+01      & 2.12e+01      & 2.12e+01     & 2.12e+01      & 2.13e+01     & 2.12e+01      \\ \hline
		\multirow{3}{*}{CPU}
%		& SART     & 5.52e+00      & 5.40e+00      & 4.91e+00     & 4.91e+00      & 6.74e+00     & 4.27e+00      \\
		& PGSA\_BE & 1.17e+03      & 1.03e+03      & 1.69e+03     & 5.16e+01      & 8.24e+02     & 5.24e+02      \\
		& e-PSG    & 1.91e+03      & 2.24e+03      & 2.24e+03     & 2.31e+03      & 2.04e+03     & 1.92e+03      \\
		& FPSA-nl  & 3.17e+01      & 1.34e+01      & 4.61e+01     & 1.42e+01      & 4.37e+01     & 1.93e+01      \\ \bottomrule
	\end{tabular}
\end{table}

 In Figure \ref{SL_FB}, we display the original images of  FB, as well as the reconstructed results via PGSA\_BE, e-PSG, and FPSA-nl under the projection angle of 150$^{\circ}$ and noise level at 0.1\%. The contrast ranges of  FB were adjusted to [0,0.65]. As can be seen, FPSA-nl almost perfectly reconstructs the original image, while the image obtained by PGSA\_BE is slightly blurry. In contrast, the image obtained by e-PSG lacks many details and contains some noisy pixels.

Table \ref{SL_tabular}  presents the SSIM and RMSE values, along with the CPU time  using SL and FB via FPSA-nl, PGSA\_BE, and e-PSG, across various angles and noise levels. Among FPSA-nl, PGSA\_BE, and e-PSG, the first two methods significantly outperform e-PSG. This is because the BC condition of e-PSG doesn't satisfy. It leads to interpolation step lengths being 0, which worsens the performance of e-PSG.

The SSIM values  from FPSA-nl and PGSA\_BE are similar, with FPSA-nl slightly outperforming PGSA\_BE. Both methods outperform e-PSG. Moreover, FPSA-nl achieves the lowest RMSE while requiring the least computational time, highlighting its efficiency among these three.

\subsection{Optimal Portfolio Selection}

We now consider the optimal portfolio selection model (\ref{OPS}). The matrix \(V \in \mathbb{R}^{n \times n}\) has the form \(V = {\Sigma} + L L^{\top}\), where \({\Sigma}\) is an \(n \times n\) diagonal matrix with positive diagonal entries and \(L\) is an \(n \times m\) matrix with \(m < n\).
The data (\(V\), \({\mu}\), \(\h d\)) were generated as follows: the diagonal entries of \({\Sigma}\) were set equal to 2, and each entry of \(L\) was uniformly drawn from \([-1, 1]\). The components of the vector \({\mu}\) were randomly generated from a uniform distribution in \((0, 1)\), which ensured that \({\mu^{\top}}\h x > 0\) for all \(\h x \in {\cal S}\). All entries of the vector \(\h d\) were set to \(1.75/n\).

We utilize FPSA-nl, e-PSG, and PGSA\_BE to solve (\ref{OPS}) and measure their performance in terms of the objective function value (ObjVal), infeasibility (Infeas:= \(|\h e^{\top}\h x - 1| + ||\max\{-\h x, \h 0\}||_1 + ||\max\{\h x - \h d, \h 0\}||_1\)), CPU time (CPU), and StatRes.
We use the stopping criterion (\ref{STOPC}) with \({\tt MaxIt} = 3000\) and \({\tt Tol} = 10^{-8}\).

The iteration for \({\h x}^{k+1}\) in FPSA-nl is

\begin{equation}
	\h x^{k+1} = \arg\min_{\h x\in \cal{S}} \frac{1}{2\delta_{k}}||\h x - \h z^{k}||^{2}  \label{OPSsub}
\end{equation}
where $\h z^{k} = \h u^{k} -\delta_{k}\nabla h(\h x^{k}) + \theta_{k}\delta_{k}\nabla^{\top}\h y^{k+1}$.
The solution of (\ref{OPSsub}) is:
\begin{equation*}
	\h x^{k+1} = \min\{\max\{ \h z^{k} - \eta^{*}_{k}\h e, \h 0\},\h d\}
\end{equation*}
$\eta^{*}_{k}$ is the root of the following one-dimension equation:
\begin{equation*}
	l(\eta) = \h e^{\top}\min\{\max\{ \h z^{k} - \eta \h e, \h 0 \},\h d\} - 1 = 0.
\end{equation*}
%hence we can employ the binary search method for solving it by the stopping criterion $|l(\eta)| < 10^{-8}$.

For FPSA-nl, we set \(\sigma = 1.05,\ \rho_{1}  = 10^{-3}\). Set $\delta_{k,0}$  as (\ref{del0})
with $\varsigma=0.82$.
Let \(q = 0.95,\ T = 20, \ N = 250\).

For PGSA\_BE, we took \(f(\h x) = \iota_{\mathcal{S}}(\h x),\ h(\h x) = \h x^{\top}V\h x,\ g(\h x) = \mu^{\top}\h x\), and we set \(\epsilon = 10^{-4},\ \alpha = 1/(2||V||), \ l = 0\). Then we calculated a recursive sequence \(\theta_{k+1} = (1 + \sqrt{1 + 4\theta^2_k})/2\), where \(\theta_{-1} = \theta_0 = 1\). We set \(\beta_k = (\theta_{k-1} - 1)/\theta_k\) and reset \(\theta_{k - 1} = \theta_k = 1\) every 100 iterations \((\beta_{100} \approx 0.97)\). Then \(\{\beta_k\} \subseteq [0, \bar{\beta}]\) for some \(0 < \bar{\beta} < 1\).

For e-PSG, we took \(f^{n}(\h x) = 0,\ f^{s}(\h x) = \h x^{\top}V\h x,\ g(\h x) = \mu^{\top}\h x\), and we set \(\beta = 0,\ \delta = 10^{-3},\ l = 2||V||_2,\ \tau_{k} \equiv 1/\delta\). Note that in this case, the BC condition was satisfied with \(m = \min_{1 \leqslant i \leqslant n} \{\mu_{i}\}\) and \(M = \max_{1 \leqslant i \leqslant n} \{\mu_{i}\}||\h d||_1\). We set \(\bar{\mu} = \frac{0.99\delta}{2}\sqrt{\frac{m}{M}}\) and \(\bar{\kappa} = 0\).

All these comparing algorithms were started from \(\h e/n\).

\begin{table}[htbp]
	\caption{The averaged results for solving optimal portfolio selection model. }
	\renewcommand{\arraystretch}{1}
	\label{OPS_200}
	\centering
	%\resizebox{1\textwidth}{0.95in}{
		\begin{tabular}{ccccccc}
			\toprule
			\multicolumn{2}{c}{n} &\multicolumn{5}{c}{n = 200}                                                               \\
			\cmidrule[\heavyrulewidth](lr){1-2} \cmidrule[\heavyrulewidth](lr){3-7}
			\multicolumn{2}{c}{m}              & m=1      & m=5      & m=20     & m=40     & m=50     \\ \toprule
			\multirow{3}{*}{ObjVal}&FPSA-nl& 1.89e-02 & 1.93e-02 & 2.08e-02 & 2.36e-02 & 2.55e-02 \\
			&PGSA\_BE& 1.89e-02 & 1.93e-02 & 2.08e-02 & 2.36e-02 & 2.55e-02 \\
			&e-PSG& 1.89e-02 & 1.93e-02 & 2.08e-02 & 2.36e-02 & 2.55e-02 \\ \hline
			\multirow{3}{*}{StatRes}&FPSA-nl& 1.84e-08 & 1.20e-06 & 9.40e-06 & 8.49e-05 & 3.47e-04 \\
			&PGSA\_BE& 1.63e-07 & 1.26e-06 & 9.40e-06 & 8.49e-05 & 3.47e-04 \\
			&e-PSG& 1.93e-07 & 1.32e-06 & 9.40e-06 & 8.49e-05 & 3.47e-04 \\ \hline
			\multirow{3}{*}{Infeas}&FPSA-nl& 4.60e-09 & 3.86e-09 & 4.12e-09 & 2.61e-09 & 2.24e-09 \\
			&PGSA\_BE& 4.40e-09 & 4.52e-09 & 3.50e-09 & 4.47e-09 & 3.56e-09 \\
			&e-PSG& 4.40e-09 & 4.52e-09 & 3.50e-09 & 4.47e-09 & 3.56e-09 \\ \hline
			\multirow{3}{*}{CPU}&FPSA-nl& 4.52e-02 & 4.56e-02 & 6.39e-02 & 7.44e-02 & 8.38e-02 \\
			&PGSA\_BE& 1.44e-01 & 1.58e-01 & 1.89e-01 & 2.66e-01 & 2.66e-01 \\
			&e-PSG& 1.68e-01 & 1.62e-01 & 1.92e-01 & 2.44e-01 & 2.56e-01 \\
			\toprule
			\multicolumn{2}{c}{n} &\multicolumn{5}{c}{n = 800}                                                               \\
			\cmidrule[\heavyrulewidth](lr){1-2} \cmidrule[\heavyrulewidth](lr){3-7}
			\multicolumn{2}{c}{m}                                 & m=4     & m=20    & m=80    & m=160   & m=200   \\ \toprule
			\multirow{3}{*}{ObjVal}&FPSA-nl& 4.70e-03 & 4.76e-03 & 5.17e-03 & 5.92e-03 & 6.55e-03 \\
			&PGSA\_BE& 4.70e-03 & 4.76e-03 & 5.17e-03 & 5.92e-03 & 6.55e-03 \\
			&e-PSG& 4.70e-03 & 4.76e-03 & 5.17e-03 & 5.92e-03 & 6.55e-03 \\ \hline
			\multirow{3}{*}{StatRes}&FPSA-nl& 1.09e-07 & 1.79e-06 & 1.44e-05 & 2.30e-04 & 1.76e-03 \\
			&PGSA\_BE& 5.04e-07 & 1.97e-06 & 1.44e-05 & 2.30e-04 & 1.76e-03 \\
			&e-PSG& 1.56e-07 & 1.80e-06 & 1.44e-05 & 2.30e-04 & 1.76e-03 \\ \hline
			\multirow{3}{*}{Infeas}&FPSA-nl& 2.86e-09 & 4.10e-09 & 5.99e-09 & 4.27e-09 & 4.10e-09 \\
			&PGSA\_BE& 4.80e-09 & 4.59e-09 & 3.71e-09 & 4.30e-09 & 2.79e-09 \\
			&e-PSG& 4.80e-09 & 4.59e-09 & 3.71e-09 & 4.30e-09 & 2.79e-09 \\ \hline
			\multirow{3}{*}{CPU}&FPSA-nl& 5.88e-01 & 7.19e-01 & 8.67e-01 & 8.66e-01 & 9.00e-01 \\
			&PGSA\_BE& 4.09e+00 & 4.29e+00 & 3.82e+00 & 4.88e+00 & 4.87e+00 \\
			&e-PSG& 3.41e+00 & 3.36e+00 & 3.39e+00 & 3.35e+00 & 3.23e+00 \\  \bottomrule
		\end{tabular}
		%}
\end{table}

 We tested different scenarios of
\begin{eqnarray*}
&&(m,n)= \nn\\&&\{(1,200), (5,200),(20,200), (40,200), (50,200), (4,800), (20,800), (80,800), (160,800), (200,800)\}
\end{eqnarray*}
and recorded the average results from 20 runs, in terms of the objective function value (ObjVal), StatRes, Infeas, and CPU time (CPU) for FPSA-nl, PGSA\_BE, and e-PSG in Table \ref{OPS_200}. As shown in the table, in all cases, FPSA-nl usually took less time to achieve almost the same ObjVal and Infeas, and much lower StatRes. We further illustrated the performance profile of FPSA-nl, PGSA\_BE, and e-PSG in terms of CPU time and StatRes in Figure \ref{PR-DW}. As mentioned before, the analysis involved 200 {\it  different instances}.

%Let $t_{p,j}$ denote a performance metric for the $j$th solver on problem $p$ where $j=1,2,3$ and $p=1,\ldots,200$.
% We calculate the ratio $r_{p,j}$ by dividing $t_{p,j}$ by the smallest value achieved by any of the $j$ solvers for problem $p$, i.e.,
%$r_{p,j}=\frac{t_{p,j} }{\min\{t_{p,j}:1\le j\le n_j\}}$.
%For a given threshold $\tau>0$, $\pi_j(\tau)$ is the ratio of the number of problems where $r_{p,j}\le\tau$ to the total number.
% This provides insight into whether solver $j$ performs within a factor of $\tau$ compared to the best-performing solver.

\begin{figure}[htbp]
	\vspace{0cm}\centering{\vspace{0cm}
		\begin{tabular}{cc}
			\includegraphics[scale = .45]{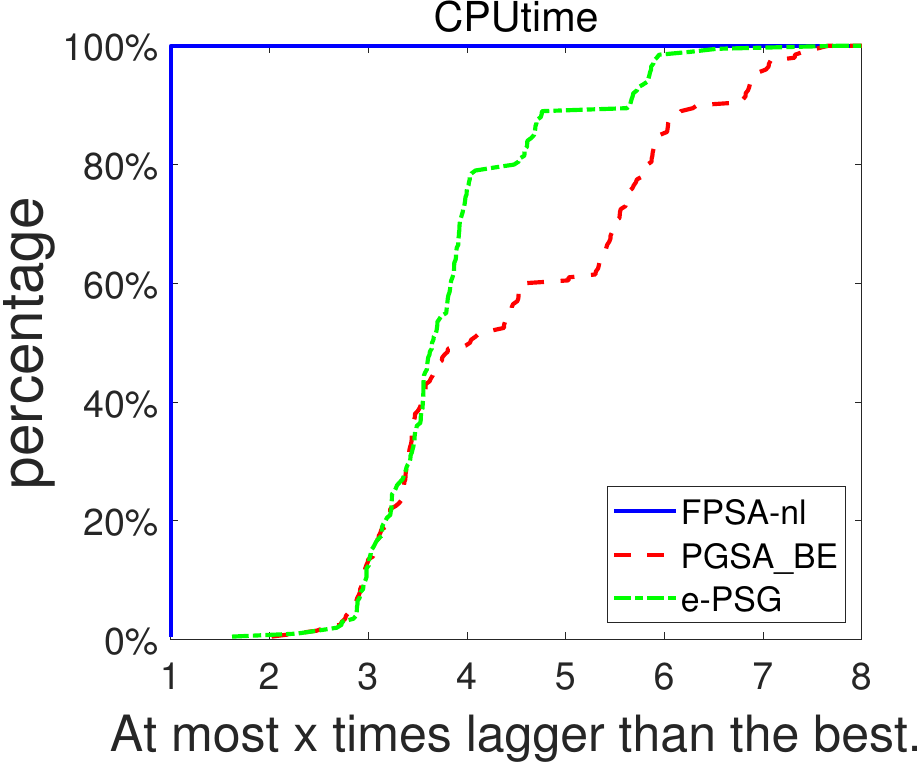}
			& \includegraphics[scale = .45]{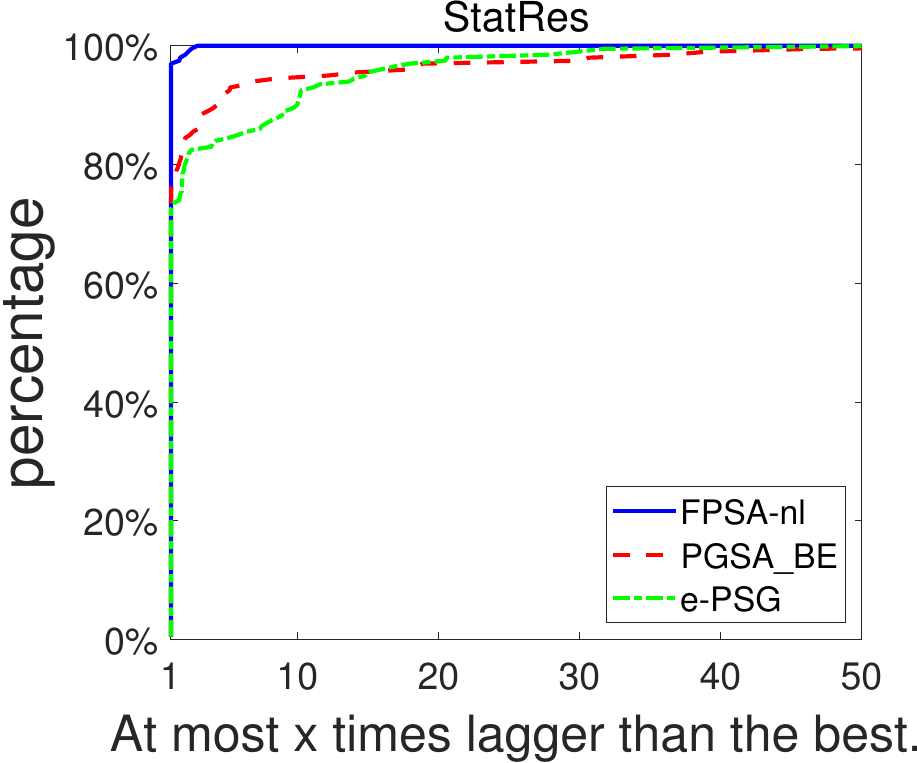}
		\end{tabular}
	}
	\caption{The performance profile of CPU time and StatRes.}
	\label{PR-DW}
\end{figure}

Figure \ref{PR-DW} presents the performance profiles for StatRes and CPU time, respectively. Notably, in each figure, the point of intersection between the `percentage' axis and the curve indicates the ratio at which the current/best solver stands among the FPSA-nl, PGSA\_BE, and e-PSG. Analysis of the curves demonstrates that for all cases, FPSA-nl has the smallest CPU time, and in nearly 80\% of the cases, the CPU time of e-PSG and PGSA\_BE is approximately at least three times that of FPSA-nl. Besides, in about 95\% of cases, the output of FPSA-nl has the lowest StatRes, demonstrating the high efficiency of FPSA-nl.

\section{Conclusions}\label{con}

This paper focused on solving a nonsmooth, nonconvex fractional programming problem with a special structure. The numerator consisted of a sum of a smooth function and a nonsmooth function, while the denominator was a composition of a nonsmooth convex function and a linear operator. We proposed a single-loop full-splitting algorithm with a relaxation step. The primary contribution of this paper was that the algorithm could achieve global convergence to an exact lifted stationary point under the suitable merit functions with the KL property. Furthermore, we combined a nonmonotone line search scheme to enhance the convergence speed and reduce the dependence on the unknown parameters. Finally, we validated the theoretical results and showcased the proposed algorithm's superiority over existing state-of-the-art methods in solving three concrete applications.

%\section*{Acknowledgments}
%\vspace{-0.4cm}
% The second author extends gratitude to Prof. Radu Ioan Bo\c{t} from the University of Vienna for his valuable discussions on Lemma 2 during her visit to him in 2023.
%
%\vspace{0.4cm}


\begin{thebibliography}{10}
\bibliographystyle{siamplain}
%\bibliography{refer5}





%
%\bibitem{Audrey15}
%{\sc A.~Repetti, M.~Q. Pham, L.~Duval, E.~Chouzenoux, and J.~C. Pesquet}, {\em
%  Euclid in a taxicab: Sparse blind deconvolution with smoothed ${L_{\nabla h}
%  _1}/{L_{\nabla h} _2}$ regularization}, IEEE Signal Process Lett., 22 (2015),
%  pp.~539--543.
%
%\bibitem{RockWets}
%{\sc R.~T. Rockafellar and R.~J.~B. Wets}, {\em Variational analysis}, 1998.
%
%\bibitem{Tao20}
%{\sc M.~Tao}, {\em Minimization of {L}$_1$ over {L}$_2$ for sparse signal
%  recovery with convergence guarantee}, SIAM J. Sci. Comput., 44 (2022),
%  pp.~A770--A797.
%
%\bibitem{Vavasis}
%{\sc S.~A. Vavasis}, {\em Derivation of compressive sensing theorems from the
%  spherical section property}, University of Waterloo,  (2009).
%
%\bibitem{WYYL20}
%{\sc C.~Wang, M.~Yan, Y.~Rahimi, and Y.~Lou}, {\em Accelerated schemes for the
%  {L}$_1$/{L}$_2$ minimization}, IEEE Trans. Signal Process., 68 (2020),
%  pp.~2660--2669.
%
%\bibitem{WangYinZeng15}
%{\sc Y.~Wang, W.~Yin, and J.~Zeng}, {\em Global convergence of {ADMM} in
%  nonconvex nonsmooth optimization}, J. Sci. Comput., 78 (2019), pp.~1--35.
%
%\bibitem{YEX14}
%{\sc P.~Yin, E.~Esser, and J.~Xin}, {\em Ratio and difference of $ L_{\nabla h}_{1} $
%  and $ L_{\nabla h}_{2} $ norms and sparse representation with coherent dictionaries},
%  Comm. Info. Systems, 14 (2014), pp.~87--109.
%
%\bibitem{ZengYuPong20}
%{\sc L.~Y. Zeng, P.~R. Yu, and T.~K. Pong}, {\em Analysis and algorithms for
%  some compressed sensing models based on {L}1/{L}2 minimization}, SIAM J.
%  Optim., 31 (2021), pp.~1576--1603.
%
%%
%% and use \bibitem to create references. Consult the Instructions
%% for authors for reference list style.
%%
%%\bibitem{RefJ}
%%% Format for Journal Reference
%%Author, Article title, Journal, Volume, page numbers (year)
%%% Format for books
%%\bibitem{RefB}
%%Author, Book title, page numbers. Publisher, place (year)

%\bibitem{}  H. Attouch and J. Bolte, {\em On the convergence of the proximal algorithm for nonsmooth functions involving analytic features,} Mathematical Programming, 116(2009), pp. 5--16.


\bibitem{ABRS} H. Attouch, J. Bolte, P. Redont, A. Soubeyran: Proximal alternating minimization and projection methods for nonconvex problems: Anapproach based on the Kurdyka-\L ojasiewicz inequality. Math. Oper. Res. 35, 438-457 (2010).

%\bibitem{ABS} H.  Attouch, J. Bolte and B.~F. Svaiter, {\em Convergence of descent methods for semi-algebraic and tame problems: proximal algorithms, forward-backward splitting, and regularized Gauss-Seidel methods}, Math. Program., 137(2013), pp. 91-129.
\bibitem{ABS} H.  Attouch, J. Bolte, B.~F. Svaiter: Convergence of descent methods for semi-algebraic and tame problems: proximal algorithms, forward-backward splitting, and regularized Gauss-Seidel methods. Math. Program. 137, 91-129 (2013).

\bibitem{banert2019general}
S. Banert, R. I. Bo\c{t}: A general double-proximal gradient algorithm for dc programming. Math. Program. 178:1, 301-326 (2019).

\bibitem{beck2017first}
A. Beck: First-order methods in optimization. Society for Industrial and Applied Mathematics, Philadelphia (2017).

%    \bibitem{BDL07}
% J.~Bolte, A.~Daniilidis and A.~Lewis, {\em The {{\L}}ojasiewicz
	%  inequality for nonsmooth subanalytic functions with applications to
	%  subgradient dynamical systems}, SIAM J. Optim., 17 (2007), pp.~1205-1223.
\bibitem{BDL07}
J.~Bolte, A.~Daniilidis, A.~Lewis: The {{\L}}ojasiewicz
inequality for nonsmooth subanalytic functions with applications to
subgradient dynamical systems. SIAM J. Optim. 17, 1205-1223 (2007).

% \bibitem{BC17} R. I. Bo\c{t} and E. R. Csetnek,  {\em Proximal-gradient algorithms for fractional programming}, Optimization 66:8 (2017), pp. 1383--1396.
\bibitem{BC17} R. I. Bo\c{t}, E. R. Csetnek: Proximal-gradient algorithms for fractional programming. Optimization 66:8, 1383--1396 (2017).
% \bibitem{APMA}
% R. I. Bo\c{t}, E. R. Csetnek, and D. K. Nguyen. \emph{A proximal minimization algorithm for structured nonconvex and nonsmooth problems.} SIAM J. Optim. 29:2 (2019), pp. 1300-1328.

% \bibitem{BDL}
% R. I. Bo\c{t}, M. N. Dao and G.~Y. Li, {\em Extrapolated proximal
	% 	subgradient algorithms for nonconvex and nonsmooth fractional programs},
% Math. Oper. Res., 47 (2022), pp.~2415-2443.
\bibitem{BDL}
R. I. Bo\c{t}, M. N. Dao, G.~Y. Li: Extrapolated proximal
subgradient algorithms for nonconvex and nonsmooth fractional programs.
Math. Oper. Res. 47, 2415-2443 (2022).

%  \bibitem{BDL23} R. I. Bo\c{t}, M. N. Dao and G.~Y. Li,  {\em Inertial proximal block coordinate method for a class of nonsmooth sum-of-ratios optimization problems}, SIAM J. Optim., 33:2 (2023), pp. 361-393.
\bibitem{BDL23} R. I. Bo\c{t}, M. N. Dao, G.~Y. Li: Inertial proximal block coordinate method for a class of nonsmooth sum-of-ratios optimization problems. SIAM J. Optim. 33:2, 361-393 (2023).

%  \bibitem{boct2023full}
% R. I. Bo\c{t}, G. Y. Li and M. Tao, \emph{A full splitting algorithm for fractional programs with structured numerators and denominators.} (2023), available at: \url{https://arxiv.org/abs/2312.14341v1}
\bibitem{boct2023full}
R. I. Bo\c{t}, G. Y. Li, M. Tao: A full splitting algorithm for fractional programs with structured numerators and denominators (2023).  \url{https://arxiv.org/abs/2312.14341v1}
%\bibitem{PADM}
%R. I. Bo\c{t} and D. K. Nguyen. \emph{The proximal alternating direction method of multipliers in the nonconvex setting: convergence analysis and rates.} Math. Oper. Res. 45:2 (2020), pp. 682-712.
%https://doi.org/10.1287/opre.28.4.927


%\bibitem{briec2004single}
%W. Briec, K. Kerstens, and J. B. Lesourd,\emph{ Single-period Markowitz portfolio selection, performance gauging, and duality: a variation on the Luenberger shortage function.} J. Optim. Theory Appl., 120 (2004), pp. 1-27.

%\bibitem{chakrabarti2021parameter}
%D. Chakrabarti, \emph{Parameter-free robust optimization for the maximum-sharpe portfolio problem.} European J. Oper. Res., 293:1 (2021), pp. 388-399.

% \bibitem{CHZ11}
%L. Chen, S. He  and S. Z. Zhang, {\em When all risk-adjusted performance measures are the same: in praise of the
	%	Sharpe ratio}, Quant. Finance, 11:10  (2011), pp. 1439-1447.
\bibitem{CHZ11}
L. Chen, S. He, S. Z. Zhang: When all risk-adjusted performance measures are the same: in praise of the
Sharpe ratio. Quant. Finance 11:10, 1439-1447 (2011).
%     \bibitem{BA20}
%R. I. Bo\c{t}, A. B{\"o}hm, {\em Variable smoothing for convex optimization problems using stochastic gradients},
% Journal of Scientific Computing,  85 (2020)  33.


%\bibitem{BN20}
%R. I. Bo\c{t}, D.-K. Nguyen, {\em The proximal alternating direction method of multipliers in the nonconvex setting: convergence analysis and rates}, Math. Oper. Res., 45 (2020), pp. 682-712.

%\bibitem{BCN19}
%R. I. Bo\c{t}, E. R. Csetnek, D.-K. Nguyen,  {\em A proximal minimization algorithm for structured nonconvex and nonsmooth problems},
% SIAM J. Optim., 29(2019), pp. 1300-1328.
% \bibitem{SB19} S. Banert, R. I. Bo\c{t}, {\em A general double-proximal gradient algorithm for d.c. programming}, Math. Program., 178:1-2
%  (2019), pp. 301--326.
%\bibitem{crouzeix1985algorithm}
%J. P. Crouzeix, J. A. Ferland, and S. Schaible. \emph{An algorithm for generalized fractional programs.} J. Optim. Theory Appl., 47:1 (1985), pp. 35-49.



%\bibitem{CP}
%P.~L. Combettes and V.~R. Wajs, {\em Signal recovery by proximal forward-backward
	%  splitting,} \emph{Multiscale Modeling $\&$ Simulation},  4 (2005), pp.
%  1168--1200.

%\bibitem{CuiPang}
%Y. Cui and J.~S. Pang,  {\em Modern Nonconvex Nondifferentiable Optimization},
%MOS-SIAM Series on Optimization, 2021.
\bibitem{CuiPang}
Y. Cui, J.~S. Pang: Modern Nonconvex Nondifferentiable Optimization.
MOS-SIAM Series on Optimization, (2021).

%\bibitem{dinkelbach1967nonlinear}
%W. Dinkelbach, \emph{On nonlinear fractional programming}, Manage.  Sci., 13:7 (1967), pp. 492-498.
\bibitem{dinkelbach1967nonlinear}
W. Dinkelbach: On nonlinear fractional programming. Manage. Sci. 13:7, 492-498 (1967).


%\bibitem{DM02}
%E.~D. Dolan and J.~J. Mor\'{e}, \emph{Benchmarking optimization software with
	%  performance profiles},  Math. Program., 91 (2002), pp. 201--213.

%\bibitem{GHW} K. Guo, D. R. Han and T. T. Wu,  \emph{Convergence of alternating direction method for minimizing sum of two nonconvex functions with linear constraints}, Int. J. Comput. Math., 94:8  (2016), pp. 1653-C1669.
\bibitem{gazzola2019ir}
S. Gazzola, P. C. Hansen, J. G. Nagy: IR Tools: a MATLAB package of iterative regularization methods and large-scale test problems. Numer. Algorithms 81:3, 773-811 (2019).



%\bibitem{hansen2012air}
%P. C. Hansen and M. S. Hansen, \emph{AIR tools-a MATLAB package of algebraic iterative reconstruction methods.} J. Comput. Appl. Math., 236:8 (2012), pp. 2167-2178.
\bibitem{hansen2012air}
P. C. Hansen, M. S. Hansen: AIR tools-a MATLAB package of algebraic iterative reconstruction methods. J. Comput. Appl. Math. 236:8, 2167-2178 (2012).

%\bibitem{he2015splitting}
%B. S. He, M. Tao and X. M. Yuan, \emph{A splitting method for separable convex programming.} IMA J. Numer. Anal., 35:1 (2015), pp. 394-426.

%\bibitem{HL93}
% J. B. Hiriart-Urruty and C. Lemar\'{e}chal, {\em Convex Analysis and Minimization Algorithms. I. Grundlehren
	%der Mathematischen Wissenschaften (Fundamental Principles of Mathematical Sciences)}, vol. 306.
%Springer, Berlin, 1993.
%\bibitem{CAAD}
%M. Y. Hong, Z. Q. Luo, and M. Razaviyayn. \emph{Convergence analysis of alternating direction method of multipliers for a family of nonconvex problems.} SIAM J. Optim. 26:1 (2016), pp. 337-364.



%\bibitem{IB83}
%T. Ibaraki,  {\em Parametric approaches to fractional programs}, Math. Program., 26 (1983), pp. 345-362 .
\bibitem{IB83}
T. Ibaraki: Parametric approaches to fractional programs. Math. Program. 26, 345-362 (1983).

%  \bibitem{JR}
%  % R. Jagannathan, \emph{On Some Properties of Programming Problems in Parametric Form Pertaining to Fractional Programming.} MANAGE. SCI. 12, no. 7 (1966): 609?15. http://www.jstor.org/stable/2627889.
%  R. Jagannathan, \emph{On some properties of programming problems in parametric form pertaining to fractional programming}, Manage. Sci., 12:7 (1966), pp. 609-615.

% \bibitem{KS01}
%  A. C. Kak and M. Slaney, {\em Principles of computerized tomographic imaging.} SIAM, 2001.
%    A. C. Kak and M. Slaney, Principles of Computerized Tomographic Imaging, SIAM, Philadelphia,
%2001.
%\bibitem{PCTI}
%A. C. Kak and M. Slaney. \emph{Principles of computerized tomographic imaging.} SIAM, 2001

%\bibitem{KAY}
%A. Y. Kruger, {\em On Fr\'{e}chet Subdifferentials},  Journal of Mathematical Sciences, 116 (2003), pp.~3325-3358.


%\bibitem{LiPong15}
%{G.~Y. Li and T.~K. Pong}, {\em Global convergence of splitting methods for
	%  nonconvex composite optimization}, SIAM J. Optim., 25 (2015), pp.~2434--2460.

%\bibitem{Lewis02}
%A. S. Lewis, {\em Active sets, nonsmoothness, and sensitivity}, SIAM J.
%Optim, vol.~13, no.~3 (2002), pp. 702-725.

%\bibitem{GCSM}
%G. Y. Li and T. K. Pong. \emph{Global convergence of splitting methods for nonconvex composite optimization.} SIAM J. Optim. 25:4 (2015), pp. 2434-2460.
%\bibitem{LSZ}
%Q. Li, L. X. Shen, N. Zhang and J.~P. Zhou,
%{\em A proximal algorithm with backtracked extrapolation for a class of structured fractional programming},
%Appl. Comput. Harmon. Anal, 56 (2022), pp. 98-122.
\bibitem{LSZ}
Q. Li, L. X. Shen, N. Zhang and J.~P. Zhou: A proximal algorithm with backtracked extrapolation for a class of structured fractional programming.
Appl. Comput. Harmon. Anal. 56, 98-122 (2022).

%\bibitem{lo2002statistics}
%A. W. Lo, \emph{The statistics of Sharpe ratios.} Financial analysts journal, 58:4 (2002), pp. 36-52.

%\bibitem{malitsky2023adaptive}
%Y. Malitsky, and K. Mishchenko. \emph{Adaptive proximal gradient method for convex optimization.} arXiv preprint arXiv:2308.02261 (2023).

%\bibitem{MNY06} B.~S. Mordukhovich, N.~M. Nam and  N.~D. Yen,  {\em Fr\'{e}chet subdifferential calculus and optimality conditions in nondifferentiable programming}, Optimization, 55:5-6 (2006), pp. 685-708.



%\bibitem{SST83} S. Schaible and T. Ibaraki, Fractional programming, European J. Oper. Res., 12 (1983),
%pp. 325--338.

%\bibitem{SY18} K. Shen and W. Yu, Fractional programming for communication systems---Part I: Power
%control and beamforming, IEEE Trans. Signal Process., 66 (2018), pp. 2616--2630.

% \bibitem{pang1980parametric}
% J. S. Pang, \emph{A parametric linear complementarity technique for optimal portfolio selection with a risk-free asset}, Oper. Res., 28:4 (1980), pp. 927-941.
\bibitem{pang1980parametric}
J. S. Pang: A parametric linear complementarity technique for optimal portfolio selection with a risk-free asset. Oper. Res. 28:4, 927-941 (1980).

%\bibitem{PR96}
%R.~A. Poliquin and R.~T. Rockafellar, {\em Prox-regular functions in variational
	%	analysis}, Trans. Amer. Math. Soc, vol. 348, no.~5 (1996), pp. 1805-1838

% \bibitem{Pang81}
%J.-S. Pang and P. S. C. Lee, {\em A parametric linear complementarity technique for the computation of equilibrium prices in a single commodity spatial model}, Math. Program., 20 (1981), pp. 81--102.


%
%\bibitem{PTSIAM}
%J.-S. Pang and M.~Tao, {\em Decomposition methods for computing directional
	%  stationary solutions of a class of nonsmooth nonconvex optimization
	%  problems,} {SIAM J. Optim.}, 28 (2018), pp.
%  1640--1669.

\bibitem{Rock70}
R.~T. Rockafellar: Convex Analysis. Princeton University Press, Princeton, (1970).


\bibitem{RockWets}
R.~T. Rockafellar, R.~J.~B. Wets: Variational Analysis. Springer, Berlin, (1998).
\bibitem{SS95} S. Schaible, Fractional programming, in Handbook of Global Optimization, R. Horst and P.M. Pardalos, eds. Nonconvex Optim. Appl. 2, Kluwer Academic Publishers, Dordrecht, 495--608 (1995).
%\bibitem{Rock70}
% R.~T. Rockafellar, {\em Convex analysis}, 1970.
%
%\bibitem{RockWets}
% R. T. Rockafellar and R. Wets, {\em Variational analysis}, 1998.

%  \bibitem{Tao22} M. Tao, {\em
	% Minimization of $L_1$
	% over $L_2$
	% for sparse signal recovery with convergence guarantee},
%  SIAM J. Sci. Comput., 44:2 (2022), pp. A770-A797.
\bibitem{Tao22} M. Tao: Minimization of $L_1$ over $L_2$ for sparse signal recovery with convergence guarantee. SIAM J. Sci. Comput. 44:2, A770-A797 (2022).

%\bibitem{van2009probing}
%E. Van Den Berg and, M. P. Friedlander \emph{Probing the Pareto frontier for basis pursuit solutions}, SIAM J. Sci. Comput., 31:2 (2009), pp.890-912.

%\bibitem{WTNL}
%C.~Wang, M.~Tao, J.~G. Nagy and Y.~Lou,  {\em Limited-angle {CT} reconstruction
	%  via the {L}$_1$/{L}$_2$ minimization,} SIAM J. Imaging Sci., 14:2 (2021), pp. 749-777.
\bibitem{WTNL}
C.~Wang, M.~Tao, J.~G. Nagy, Y.~Lou: Limited-angle {CT} reconstruction via the {L}$_1$/{L}$_2$ minimization. SIAM J. Imaging Sci. 14:2, 749-777 (2021).

%\bibitem{WBSS} Z. Wang, A.~C. Bovik, H.~R. Sheikh and E.~P. Simoncelli, {\em Image quality assessment: from error visibility to structural similarity,}  IEEE Trans. Image Process., 13:4 (2004), pp. 600-612.

%\bibitem{WNF09}
%   S. J. Wright, R. D. Nowak and M. A. T. Figueiredo, {\em Sparse Reconstruction by Separable Approximation,}  IEEE Trans. Signal Process., 57:7  (2009), pp. 2479-2493.



%\bibitem{70}
%Z. Yu, F. Noo, F. Dennerlein, A. Wunderlich, G. Lauritsch and J. Hornegger, \emph{  Simulation tools for two-dimensional experiments in x-ray computed tomography using the FORBILD head phantom}, Phys. Med. Biol., 57 (2012), pp. N237--N252.
\bibitem{70}
Z. Yu, F. Noo, F. Dennerlein, A. Wunderlich, G. Lauritsch, J. Hornegger: Simulation tools for two-dimensional experiments in x-ray computed tomography using the FORBILD head phantom. Phys. Med. Biol. 57, N237--N252 (2012).

%\bibitem{ZengYuPong20}
%L.~Y. Zeng, P.~R. Yu, and T.~K. Pong, {\em Analysis and algorithms for some
	%  compressed sensing models based on {L}1/{L}2 minimization,} SIAM J. Optim., 31 (2021), pp. 1576-1603.
\bibitem{ZengYuPong20}
L.~Y. Zeng, P.~R. Yu, T.~K. Pong: Analysis and algorithms for some
compressed sensing models based on {L}1/{L}2 minimization. SIAM J. Optim. 31, 1576-1603 (2021).

%\bibitem{ZLSIAM}
%N.~Zhang and Q.~Li, {\em First-order algorithms for a class of fractional
	%  optimization problems,} {SIAM J. Optim.}, 32 (2022),
%  pp. 100-129.
\bibitem{ZLSIAM}
N.~Zhang, Q.~Li: First-order algorithms for a class of fractional optimization problems. {SIAM J. Optim.} 32, 100-129 (2022).


\end{thebibliography}
\end{document}